\documentclass[10pt,leqno]{article}
\usepackage{amsfonts, amsmath, amsthm, amssymb}
 \usepackage{setspace}
\newtheorem{theorem}{Theorem}[section]
\newtheorem{lemma}[theorem]{Lemma}

\newtheorem{Cor}[theorem]{Corollary}

\newcommand{\hphi}{\widehat{\phi}}
\newcommand{\hT}{\widehat{T}}

\newcommand{\hx}{\widehat{x}}
\newcommand{\hU}{\widehat{U}}
\newcommand{\hV}{\widehat{V}}

\newcommand{\hnu}{\widehat{\nu}}
\newcommand{\DnT}{D^T_\nu}
\newcommand{\DnhT}{D^{\hT}_{\hnu}}
\newcommand{\tU}{\widetilde{U}}
\newcommand{\tV}{\widetilde{V}}
\newcommand{\homega}{\widehat{\omega}}

\newcommand{\wrt}{w.r.t.\ }

\def\phi{\varphi}
\def\res{\!\!\upharpoonright\!}     

\def\+{\oplus}
\def\*{\odot}

\newcommand{\0}{\mathbf{0}}

\newcommand{\ra}{\rightarrow}

\newcommand{\ba}{\begin{eqnarray*}}
\newcommand{\ea}{\end{eqnarray*}}
\newcommand{\bes}{\begin{enumerate}\topsep=1.5mm \itemsep=-1mm}
\newcommand{\ee}{\end{enumerate}}

\newcommand{\calE}{\mathcal{E}}

\newtheorem{Definition}[theorem]{Definition}

\newcommand{\ov}{\overline}

\newcommand\leT{\leq_{\rm T}}

\newcommand{\nleT}{\nleq_{\rm T}}

\def\equivT{\equiv_{\rm T}}
\newcommand{\es}{\emptyset}

\def\diverges{\!\uparrow\,}
\def\converges{\!\downarrow\,}

\newcommand{\ce}{c.e.\ }

\newcommand{\al}{\alpha}

\newcommand{\h}{\hat}
\newcommand{\wh}{\widehat}

\newcommand{\hA}{{\widehat A}}

\newcommand{\DnA}{D^A_\nu}
\newcommand{\DnhA}{D^{\hA}_{\hnu}}
\newcommand{\Fbp}{\scrF_{\beta}^+}

\newcommand{\Wfi}{W_{f(i)}}
\newcommand{\Wgi}{W_{g(i)}}

\newcommand{\scrF}{\mathcal{F}}

\def\<{\langle}
\def\>{\rangle}

\newcommand{\Vhat}{\widehat{V}}
\newcommand{\Uhat}{\widehat{U}}

\def\om{\omega}

\newcommand{\LG}{\mathcal L^\mathcal G}

\begin{document}
\title{Computably Enumerable Sets that are Automorphic to Low Sets}
\author{Peter Cholak and Rachel Epstein}

\date{} 

\maketitle


\begin{abstract}

  We work with the structure consisting of all computably enumerable (c.e.) sets
  ordered by set inclusion.  The question we will partially
  address is which c.e.\ sets are autormorphic to low (or low$_2$)
  sets.  Using work of R.\ Miller \cite{MR1933396}, we can see that every
  set with semilow complement is $\Delta^0_3$ automorphic to a low
  set. While it remains open whether every set with semilow complement
  is effectively automorphic to a low set, we show that there are sets
  without semilow complement that are effectively automorphic to low
  sets. 
  We also consider other lowness notions such as having a
  semilow$_{1.5}$ complement, having the outer splitting property,
  and having a semilow$_2$ complement.  We show that in every nonlow
  \ce degree, there are sets with semilow$_{1.5}$ complements without
  semilow complements as well as sets with semilow$_2$ complements and
  the outer splitting property that do not have semilow$_{1.5}$
  complements.
    We also address the question of which sets are
  automorphic to low$_2$ sets.

\end{abstract}

\section{Introduction}

Our domain of discourse is the collection of all c.e.\ sets under
inclusion.  This structure is called $\mathcal{E}$. By adding
intersection, union, the empty set and $\omega$, this structure is a
lattice.  We should also mention that these binary operations in addition to
$\emptyset$, $\omega$, and being computable and finite are definable in
this structure.  Thus, if we have an automorphism of $\mathcal{E}$ then
computable sets must go to computable sets and infinite sets to
infinite sets.

If we take the quotient of $\calE$ modulo the ideal of finite sets, we get the lattice $\calE^*$ of \ce
sets up to finite difference.  Soare \cite[XV 2.5]{Soare:87}
showed that if there is an automorphism of $\calE^*$ taking $A$ to
$B$, then there is one of $\calE$ as well, so we can work in
$\calE^*$ to show that automorphisms of $\calE$ exist.  

A c.e.\ set $W$ is \emph{low} if and only if its jump, $W'$, is Turing
equvialent to $\emptyset'$.  As we will discuss below, the low sets are
special within $\mathcal{E}$.  Our goal is to understand as best
possible which sets look like and behave like low sets in this
structure. That is, when does a set have a low set in its orbit?

The first result in this vein is a result of Soare \cite{Soare:82}
following \cite{Soare:74}.
Soare showed that if a \ce set $A$ is low then its collection of c.e.\
supersets of $A$ under inclusion is isomorphic to $\mathcal{E}$.
Formally, $\mathcal{L}(A) = \{ W_e \cup A \mid e \in \omega\}$ is
isomorphic to $\mathcal{E}$.  This suggests that low sets are similar to
computable sets.  Again, we can work modulo the finite sets, showing that the 
quotient of $\mathcal{L}(A)$ modulo the finite sets, called $\mathcal{L}^*(A)$, is isomorphic to $\calE^*$.  

Actually Soare proved a stronger result.  But for that we need a few
definitions.  First a set $X$ is \emph{semilow} if and only if
$\{e \mid W_e \cap X \neq \emptyset\} \leq_T \mathbf{0'}$. This index set
is $\Sigma^X_1$. If $\overline{X}$ is c.e.\ and low then this index
set is $\Delta^0_2$.  So if $A$ is low then $A$ has a semilow
complement. Now consider an isomorphism $\Phi$ between
$\mathcal{L}^*(A)$ (or $\mathcal{E}^*)$ and $\mathcal{E}^*$.  There are
functions $g$ and $h$ such that $\Phi(W_e \cup A) = W_{g(e)}$ and
$\Phi^{-1}(W_e) = W_{h(e)} \cup A$ (or $\Phi(W_e) = W_{g(e)}$ and
$\Phi^{-1}(W_e) = W_{h(e)} $).  If computable ($\Delta^0_3$) $g$ and
$h$ can be found for $\Phi$, then we call $\Phi$ \emph{effective}
($\Delta^0_3$).  It is standard to use the terms ``effective'' or ``$\Delta^0_3$'' to describe isomorphisms between $\mathcal{L}(A)$ or $\calE$ and $\calE$,
even though the isomorphisms produced may only be effective or $\Delta^0_3$ on the quotient spaces $\mathcal{L}^*(A)$ and $\calE^*$.  This is the way we will use 
these terms in this paper. 

Soare \cite{Soare:82} showed that if $A$ is c.e.\ set with semilow complement then
$\mathcal{L}^*(A)$ is effectively isomorphic to $\mathcal{E}^*$.  There have
been several improvements on this result. We observe, using
work of R.\ Miller \cite{MR1933396}, that this can be improved to if $A$ is a
c.e.\ set with semilow complement then $A$ is $\Delta^0_3$
automorphic to a low set (see Section~\ref{semilow}). 
  This means there
is a $\Delta^0_3$ automorphism $\Phi$ such that
$\Phi(A) = \widehat{A}$ is low.  One recent open question raised by
Soare (personal communication) is whether the above $\Delta^0_3$ can be
replaced by effective.  Another question raised by Soare is to
characterize the sets which are effectively automorphic to low sets. It was
thought that perhaps a characterization would be all sets which have
semilow complements.  But, in Section~\ref{notsemiloweff}, we show
that this is not the case, as there are sets without semilow complements
that are effectively automorphic to low sets.  

Soare's result about sets with semilow complements was first improved
by Maass \cite{Maass:83}.  Maass showed that for any c.e.\ set $A$
with semilow$_{1.5}$ complements that $\mathcal{L}(A)$ is $\Delta^0_3$
isomorphic to $\mathcal{E}$.  A set $X$ is \emph{semilow$_{1.5}$} if
and only if the index set $\{ e \mid W_e \cap X $ is
finite$\} \leq_1 \mathbf{\emptyset''}$.  By an index set argument, we
know that if $A$ is effectively automorphic to a low set $\widehat{A}$
it must have semilow$_{1.5}$ complement ($W_e \cap A \neq^* \emptyset$
if and only if $W_{g(e)} \cap \widehat{A} \neq^* \emptyset$ and if
$\widehat{A}$ is low then the latter is $\Pi^0_2$).

Now Harrington and Soare \cite[Section 5]{MR1640265} show that there
is a property $NL(A) $ definable in the structure $\mathcal{E}$ such
that if $NL(A)$ holds then $A$ does not have semilow complement.
Morever they showed that there is a set $A$ with semilow$_{1.5}$
complement and $NL(A)$. Thus we know that not all sets $A$ with
semilow$_{1.5}$ complement can be automorphic to a low set.  But can
they be automorphic to low$_2$ sets?

We thought we had a positive answer but $NL(A)$ for nonlow$_2$ $A$
seems to be a barrier.  $NL(A)$ does provide a barrier for the
question about whether two promptly simple sets with semilow
complements are automorphic.  This barrier is discussed in Section~5.4
of \cite{MR1640265}.  (Our situation is similar. The issue seems to be
that we cannot get $\widehat{A}$ to cover $A$ in real time.  Even
though we know a state $\nu$ is emptied into $A$, it is emptied slowly
through a series of moves into RED and BLUE sets.  On the $\widehat{A}$
side we have to match this series of BLUE and RED moves and hence we
cannot quickly cover $A$. Another problem here is that depending
on $e$ the series of RED and BLUE moves can change.  The series of
moves does not just depend on $\nu$.)

There was one more extension of Soare's and Maass's work.  Cholak
\cite{mr95f:03064} showed that if $A$ has the outer splitting property
and has semilow$_2$ complement then $\mathcal{L}(A)$ is isomorphic to
$\mathcal{E}$.  $A$ has the \emph{outer splitting property} if and
only if there are computable functions $g$ and $h$ such that for all
$e$, $W_e = W_{g(e)} \sqcup W_{h(e)}$,
$W_{g(e)} \cap \overline{A} =^* \emptyset$ and if
$W_e \cap \overline{A}$ is infinite then $W_{g(e)} \cap \overline{A} $
is nonempty.  Maass \cite[Lemma 2.3]{Maass:83} shows in a very clever
argument that if $\overline{A}$ is semilow$_{1.5}$ then $A$ has the
outer splitting property.  $X$ is \emph{semilow$_2$} if and only if
the index set $\{ e \mid W_e \cap X $ is
infinite$\} \leq_T \mathbf{\emptyset''}$.  At one time we thought we
could show that if $A$ has the outer splitting property and has
semilow$_2$ complement then $A$ is automorphic to a low$_2$ set.  Note
that if $A$ is automorphic to a low$_2$ then it has semilow$_2$
complement.

There is also some related recent work of Epstein, \cite{MR3003266}.
Epstein shows that there is a properly low$_2$ degree such that every
c.e.\ set in that degree is automorphic to a low set.  We were
wondering if that could be shown more modularly.  We wondered if there
is some property $P$ such that every set with property $P$ is automorphic
to a low set and there is some properly low$_2$ degree such that every set 
in that degree has
property $P$.  One reasonable candidate for $P$ would be having
semilow complement.  But Soare \cite[IV 4.10]{Soare:87} shows that
every nonlow degree contains a c.e.\ set whose complement is not
semilow (via a nice index set argument).  Other later results rule
out other possible $P$'s.

Downey, Jockusch, and Schupp \cite[Theorems~4.6 and 4.7]{MR3125901}
showed that every nonlow degree contains a c.e.\ $A$ without the outer
splitting property (so $A$'s complement is not semilow$_{1.5}$).  In related
results, 
 we show that every nonlow
degree contains a c.e.\ set $A$ whose complement is semilow$_{1.5}$
but not semilow (see Section~\ref{semilow15}) and a c.e.\ set $A$
whose complement is semilow$_2$ but not semilow$_{1.5}$ and has the
outer splitting property (see Section~\ref{semilow2OSP}).  We also
provide a nice index set argument that every nonlow$_2$ degree
contains a c.e.\ set $A$ whose complement is not semilow$_2$.

We should mention that it has been long known that if a degree is
nonlow$_2$ it must contain a c.e.\ set which is not automorphic to a
low$_2$ set.  Lachlan \cite{Lachlan:68} showed, using a true stages
construction, that every low$_2$ set has a maximal superset while
Shoenfield \cite{Shoenfield:76} showed that every nonlow$_2$ degree
contains a c.e.\ set with no maximal superset.  These two results have
been generalized by Cholak and Harrington \cite{mr2003h:03063}.  One
corollary of the work by Cholak and Harrington is that if $\mathbf{a}$
and $\mathbf{b}$ are two c.e.\ degrees and
$\mathbf{a''} \nleq_T \mathbf{b''}$, then $\mathbf{a}$ contains a
c.e.\ set not automorphic to anything computable from $\mathbf{b}$. It
is open if this can be improved to show that $\mathbf{a}$ contains a
c.e.\ set not automorphic to anything whose double jump is computable
from $\mathbf{b''}$.

We mention one more open related question: if $A$ is low$_2$ then is
$\mathcal{L}(A)$ isomorphic to $\mathcal{E}$?  We now know that there
is a properly low$_2$ set without the outer splitting property.  Thus,
a positive result here may not use the outer
splitting property and is very likely to use the a \emph{true stage}
construction in the style of Lachlan \cite{Lachlan:68}.

\section{C.e.\ sets with semilow complements are automorphic to low}

\label{semilow}

Soare \cite{Soare:82} showed that if $A$ is a c.e.\ set with semilow
complement then $\mathcal{L}(A)$ is effectively isomorphic to
$\mathcal{E}$.  It has been conjectured that in fact any \ce set $A$
with semilow complement is effectively automorphic to a low set.  Here
we show that it is possible to modify a proof of R. Miller
\cite{MR1933396} to show that every \ce set with semilow complement
can be taken by a $\Delta^0_3$ automorphism to a low set.

\begin{theorem}\label{Millers}[R. Miller, \cite{MR1933396}, Theorem 1.1.1] For every \ce set $A$ with semilow complement and every noncomputable \ce set
$C$, there exists a $\Delta^0_3$ automorphism of $\calE$ mapping $A$ to a set $B$ such that
$C\nleT B$.
\end{theorem}

R.\ Miller states this theorem differently, saying that $A$ is a low set instead of a set with semilow complement.  However, he mentions that the construction only requires that $A$ have semilow complement.

We modify R.\ Miller's proof to get the following theorem.

\begin{theorem}\label{semilowlow} For every \ce set $A$ with semilow complement, there exists a $\Delta^0_3$ automorphism of $\mathcal E$ mapping $A$ to a low set $B$.
\end{theorem}

\begin{proof} Here we discuss the minor modifications to the proof of
  Theorem \ref{Millers} that result in a proof of Theorem
  \ref{semilowlow}.  Because our modifications are minor and the original proof using the
  complex Harrington-Soare automorphism construction of \cite{Harrington.Soare:96}, we will not
  reproduce R.\ Miller's proof here.
Instead we briefly sketch the proof and refer the reader of this section to Theorem \ref{Millers} in \cite{MR1933396} for more details.

To prove Theorem \ref{Millers}, R. Miller builds an automorphism on a tree, as in Harrington-Soare \cite{Harrington.Soare:96}, which takes a given set $A$ to a constructed set $B$ with the desired property that $C\nleT B$.  The primary challenge of the theorem is to allow enough flow of elements into $B$ to match the flow of elements into $A$ while simultaneously restraining elements from $B$ so that $B$ cannot compute $C$.  There are two key components of the construction.  The first is a list $\LG$ that keeps track of the states of elements flowing into $A$ so that we can ensure that if infinitely many elements flow into $A$ in a given state, then infinitely many will flow into $B$ in the matching state.  The second key component is Step $\hat0$, which enumerates elements into $B$ that are in the appropriate states.  Step $\hat0$ only allows elements to enter if they are large enough to preserve a given restraint.  In our construction, we modify $\LG$ and Step $\hat0$ to reflect our new restraint, but little else is changed.

In R.\ Miller's construction, Step $\hat0$ was the only step that involved putting elements into $B$.  He needed to guarantee the proper flow of elements into $B$, while preserving restraint that would ensure that $B$ would not be able to compute $C$.  In our modified construction, we just need to preserve a different restraint.  In fact, this part of the construction can be done as in Step B of Epstein's proof in \cite{MR3003266} that there is a nonlow degree such that every set in that degree is $\Delta^0_3$ automorphic to a low set.  

In order to make this modification, we first must alter the list $\LG$.  In R.\ Miller's construction, pairs $\<\al,\hnu\>$ are added to the list whenever an element $x$ enters $A$ from the $\alpha$-state $\nu$.  (Note that $\hnu$ is the corresponding state on the $\ov B/B$ side of the construction.)  In our modification, we will instead add the triple $\<\al, \hnu, x\>$ to the list $\LG$ whenever $x$ enters $A$ from the $\alpha$-state $\nu$.  Note that we can identify each triple with a number in $\omega$.

Next, we replace R.\ Miller's Step $\hat0$ with the following new Step $\hat0$.\\

\noindent{\bf Step $\hat0$}: (Moving elements into B.)

Find the first triple $\<\al, \hnu_0, x\>$ in $\LG$ such that there is a $\wh y\in \wh\omega$ that has never before caused action on this step satisfying all of the following:

($\hat0.1$) $\alpha$ is consistent;

($\hat0.2$) $\wh y\in R_{\al, s}$;

($\hat0.3$) $\hnu(\al,\wh y, s)=\hnu_0$; and

($\hat0.4$) for all $i<\<\alpha, \hnu_0, x\>$, $\phi^B_{i}(i)[s]\converges\implies \phi^B_i(i)[s]<\wh y$.

\medskip
\noindent
{\bf Action.}  If $\<\al, \hnu_0, x\>$ is not checked, check $\<\alpha,x,e\>$, and do not enumerate $\wh y$ into $B$.    If $\<\al, \hnu_0, x\>$ has been checked already, enumerate $\wh y$ into $B$ and remove $\<\al, \hnu_0, x\>$ from the list $\LG$. This will leave infinitely many elements in $\ov B$, while still matching the flow into $A$.

The purpose of this step is essentially the same as in the original construction.  It creates a flow of elements into $B$ matching the flow into $A$, while also respecting the restraint of the negative requirements and ensuring that infinitely many elements remain outside of $B$.

Most of the lemmas in the proof of Theorem \ref{Millers} in \cite{MR1933396} could be kept exactly the same.  Lemmas 3.1.3 and 3.3.4 would need only very minor and straightforward changes, reflecting how the new construction still guarantees that $\ov B$ is infinite and that Step $\hat0$ is able to enumerate an element for every $\<\al, \hnu_0,x\>$ on $\LG$, with $\alpha$ on the true path, which works essentially the same as before, but with the old restraint replaced by the new one.

The only significant difference in the verification would be to replace Lemmas 3.3.1 and 3.3.2 with the following lemma, which is the same as Lemma 8.17 \mbox{in \cite{MR3003266}.}

\begin{lemma}\label{Blow}
The set $B$ is low.
\end{lemma}
\begin{proof}
Suppose there exist infinitely many $s$ such that $\Phi^B_i(i)[s]\converges$.  

Let $s_0$ be the least $s$ such that for all $\<\al,\hnu, x\>\leq i$, $\<\al,\hnu, x\>$ has either been removed from the list $\LG$ already or will never be removed from the list (this can happen if it is never added to the list, or if it is added but never matched).  Since each $\<\al,\hnu, x\>$ can enter $\LG$ only once, then after stage $s_0$, no $\wh y$ will enter $B$ in order to match $\<\al,\hnu, x\>$.  Let $s>s_0$ be some stage with $\Phi^B_i(i)[s]\converges$.  Then by ($\hat0.4$), nothing can enter $B$ below the use of this computation.  So $(\forall t>s)$ $[\Phi^B_i(i)[t]\converges]$.  So either $(\forall^\infty s)[\Phi^B_i(i)[s]\converges]$ or $(\forall^\infty s)[\Phi^B_i(i)[s]\diverges]$.  Thus, B is low.
\end{proof}

This completes the modification of R.\ Miller's proof to show that every \ce set with semilow complement is $\Delta^0_3$ automorphic to a low set.

\end{proof}

It remains open whether every \ce set with semilow complement is effectively automorphic to a low set.  In the next section, we show that there are sets without semilow complement that are effectively automorphic to low sets.  

\section{Effectively automorphic to low but not semilow}

\label{notsemiloweff}

In Theorem \ref{Q5}, we build an effective automorphism of
$\mathcal E^*$ that takes a set $A$ without semilow complement to a
low set $\ov A$.  To build an effective automorphism, we use Soare's
Effective Extension Theorem.

The main tool in constructing an automorphism of $\calE^*$ is matching
infinite $e$-states.  Suppose we are given listings of all the \ce
sets modulo finite difference, $\{U_e\}_{e\in\omega}$ and
$\{V_e\}_{e\in\omega}$, and an invertible map $\Theta$ of $\calE^*$
that takes $U_e$ to $\hU_e$ and $\Theta^{-1}$ takes $V_e$ to $\hV_e$.
The {\em $e$-state} of an element $x$ tells us which \ce sets $U_i$
and $\wh{V_i}$, or $\wh{U_i}$ and $V_i$ contain the element $x$, for
all $i\leq e$.  The map $\Theta$ is an automorphism of $\calE^*$ if
there are infinitely many elements in an $e$-state $\nu$ with respect
to $U_i$ and $\hV_i$ if and only if there are infinitely many elements
in the corresponding $e$-state $\hnu$ with respect to $\hU_i$ and
$V_i$.

More formally, we consider two copies of $\omega$, which we will refer
to as $\omega$ and $\homega$.  We imagine that our automorphism is
given by a permutation of $\omega$, which can be represented as a
function from $\omega$ to $\homega$.  The $e$-state of an element
$x\in \omega$ at stage $s$ is given by the triple
$\nu(e,x,s)=\<e,\sigma(e,x,s),\tau(e,x,s)\>$, where
$\sigma(e,x,s)=\{ i\leq e \mid x\in U_{i,s}\}$ and
$\tau(e,x,s)=\{ i\leq e \mid x\in \hV_{i,s}\}$.  The $e$-state of an
element $\hx\in\homega$ is determined the same way, except with
$\hU_{i,s}$ replacing $U_{i,s}$ and $V_{i,s}$ replacing $\hV_{i,s}$.
The final $e$-state of an element is
$\nu(e,x)=\<e,\sigma(e,x),\tau(e,x)\>$, where
$\sigma(e,x)=\{ i\leq e \mid x\in U_{i}\}$ and
$\tau(e,x)=\{ i\leq e \mid x\in \hV_{i}\}$, and similarly for
$\hx\in\homega$.

To see that $\Theta$ is an automorphism, it suffices to show that
\begin{eqnarray}
  \label{eq_aut}
  &(\forall \nu)(\exists^{\infty}x \in \om)[\nu(e,x) = \nu \ \wrt
    \{U_n\}_{n \in \om} \mbox{ and } \{\Vhat_n\}_{n\in \om}] \\
  &\Longleftrightarrow (\exists^{\infty}\h{y} \in \h{w})[\nu(e,\h{y}) =
    \nu \ \wrt \{\Uhat_n\}_{n\in \om} \mbox{ and } \{V_n\}_{n\in \om}].
    \nonumber
\end{eqnarray}

The theorem stated below is actually a special case of Soare's
Extension Theorem [1974].  The full version is stronger than is needed
for this paper.  
Recall that for an given enumeration of two c.e.\ sets $U$ and $V$,
$U \backslash V = \{ x | \exists s [ x \in (U_{s+1}- V_s)]\}$, $U$
\emph{before} $V$ and $U \searrow V = (U \backslash V) \cap V$, $U$
\emph{before} $U$ and \emph{then} $V$.

\begin{theorem}[The Extension Theorem](Soare[1974])\label{ET} Let $A$
  and $\hA$ be infinite \ce sets, and let $\{U_n\}_{n\in\omega}$,
  $\{V_n\}_{n\in\omega}$, $\{\hU_n\}_{n\in\omega}$, and
  $\{\hV_n\}_{n\in\omega}$ be computable arrays of \ce sets.  Suppose
  there is a simultaneous computable enumeration of all of the above
  such that $\hA\searrow \hU_n=\es=A\searrow \hV_n$ for all $n$.  For
  each full $e$-state $\nu$, define

$$D_\nu^A=\{x\ : \ x\in A_{\text{at } s+1} \text{ and } \nu=\nu(e,x,s) \text{ w.r.t. } \{U_{n,s}\}_{n,s\in\omega} \text{ and } \{\hV_{n,s}\}_{n,s\in\omega}\}$$

and

$$\DnhA=\{\hx\ : \ \hx\in \hA_{\text{at } s+1} \text{ and } \hnu=\nu(e,\hx,s) \text{ w.r.t. } \{\hU_{n,s}\}_{n,s\in\omega} \text{ and } \{V_{n,s}\}_{n,s\in\omega}\}.$$

$\DnA$ is the set of all elements that enter $A$ from $e$-state $\nu$,
and similarly for $\DnhT.$

Suppose that the simultaneous enumeration satisfies:

\begin{equation}\label{DnA}  
  (\forall \nu)[\DnA \text{ infinite}\iff  \DnhA \text{ infinite}].
\end{equation}

Then there is a computable enumeration of \ce sets
$\{\tU_n\}_{n\in \omega}$ and $\{\tV_n\}_{n\in\omega}$, where $\tU_n$
extends $\hU_n$ and $\tV_n$ extends $\hV_n$ such that
$\hU_n=\tU_n\backslash \hA$, $\hV_n=\tV_n\backslash A$, and
\begin{eqnarray}\label{auto}
  \nonumber \exists^\infty x \in A \text{ with final $e$-state $\nu$ w.r.t. } \{U_{n,s}\}_{n,s\in\omega} \text{ and } \{\tV_{n,s}\}_{n,s\in\omega}\\
  \iff\quad\quad\quad\quad\quad\quad\quad\quad\quad\quad\quad\quad\quad\\
  \nonumber \exists^\infty \hx \in \hA \text{ with final $e$-state $\hnu$ w.r.t. } \{\tU_{n,s}\}_{n,s\in\omega} \text{ and } \{V_{n,s}\}_{n,s\in\omega}.
\end{eqnarray}
\end{theorem}

A {\em skeleton} of the \ce sets is a simultaneous computable
enumeration of \ce sets such that every \ce set is finitely different
from some set on the list.  Let $\{U_n\}_{n\in\omega}$ and
$\{V_n\}_{n\in\omega}$ be skeletons of the \ce sets.  Thus, if we
begin with a partial automorphism of $\calE^*$ on the complements of
$A$ and $\widehat A$ that takes $U_n$ to $\hU_n$ and $\hV_n$ to $V_n$,
then we can extend it to an automorphism of $\calE^*$ that takes $A$
to $\widehat A$, $U_n$ to $\tU_n$, and $\tV_n$ to $V_n$.

The way we build a partial automorphism on the complements of
$A$ and $\hA$ is to match infinite $e$-states for elements in $\ov A$
and $\ov \hA$.  That is, if there are infinitely many elements
$x\in \ov A$ such that $\nu(e,x)=\nu$, then there are infinitely many
elements in $\wh y\in \ov \hA$ such that $\hnu(e,\wh y)=\hnu$, and
vice versa.

We call the state $\nu$ ($\hnu$) a {\em gateway state} if $\DnT$
($\DnhT$) is infinite.  In the full version of the Extension Theorem,
we do not need an exact matching of gateway states $\nu$ and $\hnu$,
but only a {\em covering} of the states, as described in Soare [1974].
Our construction gives an exact matching of gateway states, so we have
stated only the special case of the theorem here.

\begin{theorem}[Cholak/Epstein]\label{Q5} There is a set $A$ that does
  not have semilow complement, but is effectively automorphic to a low
  set.
\end{theorem}

\begin{proof}

  We will construct a \ce set $\overline{A}$ that is not semilow and a
  \ce set $\widehat A$ that is low, such that $A$ and $\widehat A$ are
  effectively automorphic to each other.  We use Soare's Effective
  Extension Theorem.

  \subsection{Requirements}
  
  Our first requirement is the automorphism requirement, which constructs an
  automorphism taking $A$ to $\widehat A$.  To do this, we build a
  simultaneous computable enumeration of \ce sets $\{U_n\}_{n\in\omega}$,
  $\{V_n\}_{n\in\omega}$, $\{\hU_n\}_{n\in\omega}$, and
  $\{\hV_n\}_{n\in\omega}$ as in the extension
  theorem, such that $\{U_n\}_{n\in\omega}$ and 
  $\{V_n\}_{n\in\omega}$ are skeletons.  We will ensure that the map taking $U_n$ to $\hU_n$ and $\hV_n$
  to $V_n$ gives a partial automorphism of $\mathcal E^*$ on the
  complements of $A$ and $\hA$, and that we have equality of gateway
  states to meet the hypotheses of the Extension Theorem.

  Let $\phi_i$ be a listing of all partial computable functions of two
  variables.  We will use it to list all $\Delta_2$ functions.

  To achieve that $A$ does not have semilow complement, we meet the
  following requirement for all $i\in\omega$:

  P$_i$:\quad $\widehat\phi_i(e):= \lim_s \phi_i(e,s)$ is not the
  characteristic function of $\{e\ :\ W_e\cap \ov{A}$ nonempty$\}$.

  To this end, for each $i\in\omega$ we will build sets $S_k^i$ for
  each $k\in \{0,1,\ldots 4^{i+1}\}$, such that for some $k$, $S_k^i$
  will be the \ce set $W_e$ and
  $\widehat \phi_i(e)=\lim_s \phi_i(e,s)$ will be wrong about whether
  $W_e$ and $\ov{A}$ have nonempty intersection.  The index $e$ for
  $S_k^i$ may be found computably in $k$, $i$, and the number of times
  $n$ we have started over in building $S_k^i$.  We call this
  computable function $g(i,k,n)$.  By the Recursion Theorem with
  Parameters (see \cite[II 3.5]{Soare:87}), we may assume that we know the function $g$ in advance.
  We can rewrite P$_i$ as:

  P$_i$:\quad
  $(\exists k)(\exists n)[0\leq k\leq 4^{i+1}\text{ and
  }\hphi_i(g(i,k,n))=1\implies S_k^i\cap \ov A= \es, $
  and $\hphi_i(g(i,k,n))=0\implies S_k^i\cap \ov A \neq \es$, for
  $S_k^i$ having been reset $n$ times$]$.

  The number $4^{i+1}$ appears as it is the number of $i$-states.
  Since each $i$-state is determined by a subset of
  $\{U_0, \ldots, U_i\}$ as well as a subset of
  $\{\hV_0, \ldots, \hV_i\}$, there are
  $2^{i+1}\cdot 2^{i+1}=4^{i+1}$.  This will be important for matching
  entry states, as we will see in the next
  section.  

  To achieve that $\hA$ is low, we meet the usual requirement for all
  $j\in\omega$:

  N$_j$:\quad
  $(\exists^\infty s)\Phi^{\widehat A}_j(j)[s]\converges\implies
  \Phi^{\widehat A}_j(j)\converges$.

  This guarantees that $\hA$ is low because it makes the jump of $\hA$
  limit computable, and thus computable in $\bf 0'$.

  \subsection{Basic strategy}

  Our construction will take place on two identical pinball machines,
  $M$ and $\widehat M$, see Cholak \cite{mr95f:03064}, Soare
  \cite{Soare:87} or Harrington and Cholak
  \cite{Harrington.Soare:96}. Each element of $\omega$ will flow
  through $M$, and each element of a copy of $\omega$ called
  $\widehat \omega$ will flow through $\widehat M$.  The construction
  of the pinball machines is extremely simple.  They each have a
  single corridor along which all elements flow.  Along the corridor
  are gates $G_e$ for each $e\in \omega$.  Elements may be held at
  gates or pass through, according to the construction.  In our
  construction, a {\em closed} gate will let nothing through, while an
  {\em open} gate will let all elements through except for those
  currently designated as witness, as explained later.  Throughout the
  construction, we will match elements in $\omega$ to elements in
  $\homega$ by a matching function $m(x)$ with domain $\omega$, where
  we begin with $m(x)=\widehat x$, the copy of $x$ in
  $\widehat \omega$.  During the construction, elements may be
  rematched.

  Let $\{W_e\}_{e\in\omega}$ be a standard enumeration of the \ce
  sets.  We build two skeletons $\{U_e\}_{e\in\omega}$ and
  $\{V_e\}_{e\in\omega}$ as follows: If $x\in W_e$ and $x$ is either
  at gate $G_{e'}$ for $e'\geq e$ or $x$ has been removed from $M$,
  enumerate $x$ into $U_e$.  Similarly, if $x\in W_e$ and $\hx$ is
  either at gate $G_{e'}$ for $e'\geq e$ or $\hx$ has been removed
  from $\widehat M$, enumerate $\hx$ into $V_e$.  For each $e$, only
  finitely many elements never reach gate $G_e$ or leave the machines,
  so $U_e=^* W_e$ and $V_e=^*W_e$.

  Elements move through $M$ by flowing through the machine, starting
  at gate $G_0$, then moving to $G_1$, and so on.  Every element on
  the machine is at a gate.  As $x$ moves, $m(x)$ copies its move on
  $\widehat M$.  To meet the automorphism requirement, if $x$ is in
  $U_e$ while at gate $G_{e'}$ for $e'\geq e$, enumerate $m(x)$ into
  $\widehat U_e$.  Similarly, if $m(x)$ is in $V_e$, enumerate $x$
  into $\widehat V_e$.

  We define for each stage $s$ a restraint
  function
  $$r(j,s)=\max\{\phi^{\widehat A}_{j'}(j')[s]\ | \ j'\leq j\}.$$ That
  is, $r(j,s)$ is the maximum use of any jump computation
  $\Phi^\hA_{j'}(j')[s]$ for $j'\leq j$.
 
  To meet P$_i$, we would like to build a \ce set $S$ such that if
  $\phi_i$ guesses that $S\cap \ov A$ is empty, we put an element not
  currently in $A$ into $S$, and if it later guesses that
  $S\cap \ov A$ is nonempty, we put that element into $A$.  This is
  the standard method of constructing a set that does not have semilow
  complement.  It is also frequently used to build a nonlow set, as
  any set without semilow complement is not low.  In order to meet the
  automorphism requirement, if we put infinitely many elements into
  $A$, we must put infinitely elements into $\widehat A$ from the same
  $i$-state.  We cannot simply enumerate $\widehat x$ into
  $\widehat A$ whenever we enumerate $x$ into $A$ because we need to
  ensure that $\widehat A$ is low.  Instead, when we put an element
  into $A$, we would like to put a large enough element into $\hA$
  that is in the same $i$-state.  The way we will accomplish this is
  to build several \ce sets $S_k^i$ for each $i$ and we will guarantee
  that one will act as the desired set $S$.  When we want to enumerate
  an element $x$ from some $S_{k_0}^i$ into $A$ to satisfy a positive
  requirement, we will simultaneously enumerate a large enough element
  $\widehat z$ into $\hA$, where $\widehat z$ is in the same $i$-state
  as $x$.  Since we have removed $x$ and $\widehat z$ from the
  machines, we have left $m(x)$ and $m^{-1}(\widehat z)$ without
  reasonable partners, and so we partner them with each other.  By
  this process, whenever we enumerate an element into $A$ from some
  $i$-state, we also enumerate an element into $\hA$ from the same
  $i$-state, so that we achieve exact matching of entry states, as
  desired by the Effective Extension Theorem.
  
  For every $i$, we build $S_k^i$ for each
  $k\in \{0,1,\ldots 4^{i+1}\}$.  We will start with all sets empty.
  If $\phi_i$ guesses that they all have empty intersection with
  $\overline A$, then we will close gate $G_i$ and begin to fill each
  $S_k^i$ with a single element, in increasing order of $k$.  We fill
  $S_k^i$ with the least element $x$ at gate $G_i$ such that
  $m(x)>r(j,s)$ for $j=i+k$ and $m(x)$ is greater than any element in
  any $S_{k'}^i$ for
  $k'<k$.  

  When a set $S_k^i$ contains an element not in $A$, we call that
  element $x$ and its partner $m(x)$ {\em witnesses}.  Once each
  $S_k^i$ contains an element, we reopen the gate $G_i$ to all
  non-witnesses.  

  We then wait until $\phi_i$ guesses that each $S_k^i$ has nonempty
  intersection with $\overline A$.  Now, there are $4^{i+1}+1$
  witnesses and only $4^{i+1}$ different
  $i$-states.
  Thus, by the pigeonhole principle, there are two witnesses in the
  same $i$-state.  Say they are $x_0$ in $S_{k_o}^i$ and $x_1$ in
  $S_{k_1}^i$, with $k_0<k_1$.  We will enumerate $x_0$ into $A$ and
  $m(x_1)$ into $\widehat A$.  These elements are entering in the same
  state.  Once elements enter $A$ or $\widehat A$, they are removed
  from the pinball machine.  Their previous matches,
  $m(x_0)=\widehat y$ and $x_1$, do not get removed from the pinball
  machine.  Instead, set $m(x_1)$ to be $\widehat y$, the element
  previously known as $m(x_0)$.  We reset all $S_{k'}^i$ for $k'>k_0$,
  removing the name ``witness'' from any elements in sets that are
  reset, and starting new empty $S_{k'}^i$.  Note that $S_{k_0}^i$ has
  not been reset, but its only element has entered $A$.

  If ever $\phi_i$ again guesses correctly which sets $S_k^i$ have
  nonempty intersection with $\ov A$, then we must repeat the process,
  with some minor changes.  We again close the gate $G_i$ until there
  is a single element in each $S_k^i$ that is not also in $A$.  For
  $k=0$, we fill $S_0^i$ as before, with the least element $x$ such
  that $m(x)>r(j,s)$ for $j=i+0=i$.  For $k>0$, we wait until
  $S_{k-1}^i$ has been filled and enumerate into it the least $x$ such
  that $m(x)>r(j,s)$ for $j=i+k+n^i_k$, where $n^i_k$ is the number of
  times $S_k^i$ has been reset by the action of
  P$_i$. 
  The purpose of $n^i_k$ is so that if P$_i$ acts infinitely often, it
  will still only injure each N$_j$ finitely often.  Once each $S_k^i$
  contains a witness, we reopen $G_i$ to all non-witnesses.  We then
  continue as in the previous paragraph.

  In order that $S_k^i$ respect N$_j$ whenever $j\leq i+k+n^i_k$, if
  $\Phi^{\widehat A}_j(j)[s+1]\converges$ by a new computation, we
  reset $S_k^i$.  This causes any witnesses to no longer be witnesses
  and the set to be built again from an empty set.  Note that it does
  not cause $n_k^i$ to increase because the resetting was caused by
  N$_j$ and not by the action of P$_i$.

  The primary reason for using multiple $S_k^i$ sets instead of a
  single $S^i$ is that we want to respect more and more computations
  when we enumerate elements into $\hA$, which we could not do with
  only a single $S^i$ that we can only reset finitely often.  We will
  show in the verification that there will be some $S_k^i$ that is
  only reset finitely often that we will use to satisfy $P_i$.

  \subsection{Construction}


  {\em Stage $s=0$.}  Let each $n_k^i=0$.  All sets begin empty and
  all gates begin open.

  {\em Stage $s+1$}.

  {\bf Step 1}: Place $s$ on machine $M$ at gate $G_0$ and
  $\widehat s$ on machine $\widehat M$ at gate $G_0$.  Let
  $m(s)=\widehat s$.

  {\bf Step 2}: For each $x$ and $e$, if $x\in W_{e,s}$, and $x$ is
  either at gate $G_{e'}$ for $e'\geq e$ or $x$ has been removed from
  the machine, enumerate $x$ into $U_{e,s+1}$.  Similarly, if
  $\hx\in W_{e,s}$, and either $\hx$ is at gate $G_{e'}$ for
  $e'\geq e$ or $\hx$ has been removed from the machine, enumerate
  $\hx$ into $V_{e,s+1}$.
 
  In addition, if $x\in U_{e,s+1}$ and is still on the
  machine, 
  enumerate $m(x)$ into $\widehat U_{e,s+1}$.  Similarly, if
  $\widehat x\in V_{e,s+1}$ and is still on the
  machine, 
  enumerate $m^{-1}(\hx)$ into $\widehat V_{e,s+1}$.

  {\bf Step 3}: If $\Phi_j^{\widehat A}(j)[s]\converges$ via a new
  computation, reset all $S_k^i$ with $k+i+n_k^i\geq j$ by creating a
  new $S_k^i$, increasing $n_k^i$ by one, and cancelling all witnesses
  from $S_k^i$.

  {\bf Step 4}: For each $i$, in increasing order, check if either of
  the following cases apply.

  \quad {\em Case 4A} (filling $S_k^i$): At least one $S_k^i$ has
  empty intersection with $\ov A$, and either the gate is closed or
  $\phi_i(g(i,k,n_k^i),s)$ is equal to the characteristic function of
  $S_k^i\cap \ov A$ for each $k$.  Close the gate if not already
  closed.  For the least $k$ such that $S_k^i\cap \ov A$ has empty
  intersection, check if there is an $x$ at gate $G_i$ such that both
  $x$ and $m(x)$ are larger than any current or previous witness at
  the gate and $m(x)>r(j,s)$ for $j=i+k+n_k^i$.  If so, enumerate $x$
  into $S_k^i$ and call it and $m(x)$ witnesses.  If we enumerated an
  element into $S_k^i$ for $k=4^{i+1}$, open the gate.  Continue to
  Step 5.

  \quad {\em Case 4B} (enumerating into $A$ and $\widehat A$): All
  $S_k^i$ have nonempty intersection with $\ov A$ and
  $\phi_i(g(i,k,n_k^i),s)=1$ for all $k$.  Find the least witness
  $x_0$ such that there is a witness $x_1$ in the same $i$-state as
  $x_0$.  Such a pair is guaranteed by the pigeonhole principle.  Say
  the element $x_0$ is in $S_{k_0}^i$ and $x_1$ is in $S_{k_1}^i$.
  Note that $k_0<k_1$.  Enumerate $x_0$ into $A$ and $m(x_1)$ into
  $\widehat A$ and remove both $x_0$ and $m(x_1)$ from the machines.
  Now $x_1$ and $\widehat y=m(x_0)$ do not have matches on the
  machine, so set $m(x_1)=\widehat y$.  Reset each $S_k^i$ for $k>k_0$
  by increasing $n_k^i$ and starting $S_k^i$ over with no witnesses.
  Note that $S_{k_0}^i$ does not get reset and so is not empty, but
  now has empty intersection with $\ov A$ and thus no witnesses.

  {\bf Step 5}: For each $x$ at an open gate $G_e$ such that $x$ is
  not a witness and $x,~ m(x)>e$, move $x$ and $m(x)$ to gate
  $G_{e+1}$ on their respective machines.

  \subsection{Verification}

  \begin{lemma}\label{injury1} Each N$_j$ is injured finitely often
    and is satisfied, so $\widehat A$ is low.
  \end{lemma}
  \begin{proof}
    Induct on $j$.  Suppose true for all $j'<j$.  Then since N$_{j'}$
    is injured finitely often, $\lim_s r(j',s)$ exists.

    If $\Phi_j^{\widehat A}(j)[s]\converges$, then the only $S_{k}^i$
    that can injure the computation by enumerating $m(x)$ into
    $\widehat A$ for $x\in S_k^i$ must satisfy $i+k<j$.  Each of these
    $S_k^i$ gets reset finitely often by N$_{j'}$, for $j'<j$, by the
    induction hypothesis.  By the construction, after an element
    $m(x)$ enters $\widehat A$, $S_k^i$ gets reset and $n_k^i$
    increases by one, as $S_k^i$ is playing the role of $S_{k_1}^i$ in
    the construction, where $k_1>k_0$.  Thus, after finitely many
    elements of the form $m(x)$ with $x\in S_k^i$ enter $\hA$, any
    future elements $x\in S_k^i$ must satisfy $m(x)>r(j,s)$ since
    $i+k+n^i_k>j$ for large enough $n^i_k$.  Since N$_j$ can only be
    injured finitely often by finitely many $S_k^i$, N$_j$ will
    eventually be satisfied.  Thus, the jump of $\widehat A$ is limit
    computable, so $\widehat A$ is low.
  \end{proof}

  \begin{lemma}\label{gateinf} For each gate $G_i$, infinitely many
    elements pass through gate $G_i$ to gate $G_{i+1}$.
  \end{lemma}

  \begin{proof} Induct on $i$.  Suppose true for all $i'<i$.  Then
    infinitely many elements reach gate $G_i$.

    Note that if gate $G_i$ is ever closed, then it will reopen when
    an element enters $S_k^i$ for $k=4^{i+1}$.  First we show that it
    will eventually reopen.  As $G_i$ is closed, we know that
    $S_{4^{i+1}}^i$ has empty intersection with $\ov A$, so Case 4A
    applies.  While the gate is closed, Case 4B will not act for this
    $i$, so the only way for any $S_k^i$ to be reset is for
    $\Phi_j^\hA(j)[s]$ to converge via a new computation for
    $j\leq k+i+n^i_k$.  However, no $n^i_k$ will change for any $k$
    while $G_i$ is closed, and Lemma \ref{injury1} tells us that
    $\Phi_j^\hA(j)[s]$ can only converge via a new computation
    finitely often, so each $S^i_k$ will only be reset finitely often
    while $G_i$ remains closed.  Furthermore, for $j=i+k+n^i_k$,
    $r(j,s)$ will eventually stop increasing while the gate remains
    closed.  As infinitely many elements arrive at gate $G_i$, each
    $S^i_k$ will eventually receive an element, including
    $S^i_{4^{i+1}}$, at which point the gate is
    opened. 
    Therefore, whenever gate $G_i$ is closed, it is eventually
    reopened.

    If the gate $G_i$ is closed only finitely often, then almost all
    elements that enter $G_i$ also enter $G_{i+1}$.  Consider the case
    that $G_i$ is closed infinitely often.  Then it is opened
    infinitely often, meaning that infinitely often, all $S_k^i$ have
    nonempty intersection with $\ov A$.  In order for $G_i$ to become
    closed again, some $S_k^i$ must have empty intersection with
    $\ov A$.  This can only happen by Step 4 Case 4B or Step
    3 applies.  In either case, there must be some $S_k^i$ that gets
    reset and so its witness is no longer a witness, and on Step 5, it
    moves to gate $G_{i+1}$ if large enough.  Thus, infinitely often,
    an element moves from $G_i$ to $G_{i+1}$.
  \end{proof}

  \begin{lemma}\label{skeleton} For each $e$, $W_e=^* U_e=^* V_e$, and
    thus $\{U_e\}_{e\in\omega}$ and $\{V_e\}_{e\in \omega}$ are
    skeletons.
  \end{lemma}

  \begin{proof} We will prove that $W_e=^* U_e$.  The proof that
    $W_e=^* V_e$ is similar.  Note that $U_e\subseteq W_e$, as we
    never enumerate an element into $U_e$ until after it appears in
    $W_e$.  Note also that the only elements in $W_e$ that are not
    enumerated into $U_e$ are those that are permanently held at a
    gate $G_i$ for $i<e$.  We know from Lemma \ref{gateinf} that all
    gates are open at infinitely many stages, thus no element is held
    forever at gate $G_i$ unless it is forever labeled a witness or if
    either $x$ or $m(x)$ are less than or equal to $i$.  The latter
    situation happens for finitely many elements.  When any $S_k^i$ is
    reset, its witnesses are cancelled, meaning they are no longer
    considered to be witnesses, so any element that is forever a
    witness is in the final incarnation of $S_k^i$ and never enters
    $A$.  (For the $V_e$ case, these witnesses are of the form $m(x)$
    for $x\in S_k^i\cap \ov A$.)  
There is never more than one element
    of $S_k^i\cap \ov A$, for each $i<e$, so there are finitely many
    elements that are forever kept at $G_i$ as witnesses.  Thus, all
    other elements of $W_e$ eventually either leave the machine or
    reach $G_e$ and are enumerated into $U_e$.
  \end{proof}

  \begin{lemma} Requirement P$_i$ is met, so $\ov A$ is not semilow.
  \end{lemma}
  \begin{proof}
    We must show that
    $\lim_s \phi_i(e,s)\neq \{e\ | \ W_e\cap \ov A\neq \es\}$.
    Suppose for a contradiction that it is false.  Let $i$ be the
    least such that $\lim_s \phi_i(e,s)$ gives the characteristic
    function of the set of $e$ such that $W_e$ has nonempty
    intersection with $\ov A$.

    Note that $S_0^i$ can only be reset by N$_j$ for $j\leq i$, which
    will only reset $S_0^i$ finitely often.  Let $k^*$ be the greatest
    $k$ such that $S_k^i$ is only ever reset finitely often.  If
    $k^*=4^{i+1}$, then Step 4 only acts finitely often, meaning that
    Case 4A and Case 4B only apply finitely often, so for almost all
    $s$, $\phi_i$ guesses wrong about whether some $S_k^i$ has
    nonempty intersection with $\ov A$.  Thus, since we are assuming
    that $\phi_i$ is eventually correct, we cannot have that
    $k^*=4^{i+1}$.

    Now, let $S_{k^*}^i$ be reset $n$ times in total.  Then
    $\lim_s \phi_i(g(i,k^*,n),s)$ is 1 if $S_{k^*}^i$ has nonempty
    intersection with $\ov A$ and 0 otherwise.  Since
    $k^*\neq 4^{i+1}$, $S_{k^*+1}^i$ gets reset infinitely often,
    which means that Step 4 Case 4B acts infinitely often by
    enumerating an element $x\in S_{k^*}^i$ into $A$.  Thus,
    infinitely often, $S_{k^*}^i$ gets a new witness in Case 4A and
    then has that witness enumerated into $A$ in Case 4B.  Therefore,
    there are infinitely many stages when $S_{k^*}^i$ has empty
    intersection with $\ov A$ and Case 4A applies because
    $\phi_i(g(i,k^*,n),s)=0$ and there are infinitely many stages when
    $S_{k^*}^i$ has nonempty intersection with $\ov A$ and Case 4B
    applies because $\phi_i(g(i,k^*,n),s)=1$.  This contradicts the
    assumption that $\lim_s \phi_i(g(i,k^*,n),s)$ exists.  Therefore,
    requirement P$_i$ is met, and since
    $\{e\ | \ W_e\cap \ov A\neq \es\}$ is not limit computable,
    $\ov A$ is not semilow.

  \end{proof}

  \begin{lemma} There is an effective automorphism of $\mathcal E^*$
    taking $A$ to $\widehat A$.
  \end{lemma}
  \begin{proof} Note that the construction provides a simultaneous
    computable enumeration of the sets $\{\widehat U_e\}_{e\in\omega}$
    and $\{\widehat V_e\}_{e\in \omega}$.  We guarantee that $\nu$ is
    infinite on the complement of $A$ if and only if $\widehat \nu$ is
    infinite on the complement of $\wh A$ by always enumerating
    elements simultaneously with their partners: $x$ and $m(x)$ are in
    the same $e$-state, where $G_e$ is the highest gate they reach.
    Note that there is no possibility for either $x$ or $\hx$ to be
    rematched infinitely often, as they each only reach finitely many
    gates and can only be witnesses once at each gate.  Also note that
    if $x$ and $m(x)$ are permanently partners and remain forever at
    gate $G_e$, then they are in the same $e'$-state for all $e'>e$ as
    well, since neither are enumerated into $U_i$, $V_i$, $\hU_i$ or
    $\hV_i$.

    Next, we apply the Effective Extension Theorem.  It suffices to
    show that the gateway states for $A$ equal the gateway states for
    $\widehat A$.  We only ever enumerate into either $A$ or
    $\widehat A$ during Step 4 Case 4B.  When we do, we pick witnesses
    $x_0$ and $x_1$ at Gate $G_i$ that are in the same $i$-state,
    which are guaranteed to exist by the pigeonhole principle.  Then,
    we enumerate $x_0$ into $A$ and $m(x_1)$ into $\widehat A$.  If
    $x_0$ is in $i$-state $\nu$, then $m(x_1)$ is in $\widehat \nu$.
    As above, note that $x_0$ and $m(x_1)$ are also in the same
    $e$-state for any $e$ because if $e>i$, then neither element is in
    $U_e$, $V_e$, $\widehat U_e$, or $\widehat V_e$.  Thus, we get
    that the gateway states $D_\nu^A$ equal the gateway states
    $\DnhA$.

    By the Effective Extension Theorem, we can extend the partial
    automorphism on the complements of $A$ and $\widehat A$ to a total
    automorphism of $\mathcal E^*$ taking $A$ to $\widehat A$.

  \end{proof}
\end{proof}


\section{Semilow$_{1.5}$ not semilow in all nonlow degrees}

\label{semilow15}

In Theorem \ref{Q5}, we showed that there is set $A$ that does not
have semilow complement and is effectively automorphic to a low set.
Since $A$ is effectively automorphic to a low set, $A$ has
semilow$_{1.5}$ complement (see the introduction).

In Theorem \ref{onefivenotsemilow}, we show that in every nonlow \ce
degree, there is a set that has semilow$_{1.5}$ complement but that
does not have semilow complement.  It remains open whether we can find
such sets in every degree that are also effectively automorphic to low
sets.

\begin{theorem}\label{onefivenotsemilow} For every nonlow \ce degree
  $\bf d$, there is a \ce set $A\in \bf d$ such that $A$ has
  semilow$_{1.5}$ complement, but does not have semilow complement.
\end{theorem}

\begin{Cor} The nonlow \ce degrees are precisely the degrees of \ce sets that have semilow$_{1.5}$ complement but not semilow complement.
\end{Cor}

The corollary follows immediately from the theorem, since every low \ce set has semilow complement.

\begin{proof}

  Given a nonlow \ce set $D$, we will construct an $A\equivT D$ such
  that $\ov A$ is semilow$_{1.5}$ but not semilow.

  \subsection{Requirements:}

  To ensure that $D\leT A$, we will code $D$ into $A$.  To begin, we
  construct a computable list of disjoint finite sets
  $F_k\subset \omega$, with $F_0$ containing $1$ element and each
  other $F_k$ having $2k^2$ elements.

  C$_k$: $k\in D$ if and only if some element of $F_k$ is in $A$.

  To make $\ov A$ not semilow, we proceed similarly as in Theorem
  \ref{Q5} by meeting the following requirement for each total
  computable function $h$ on two inputs:

  P$_h$: $\displaystyle \lim_s h(e,s)$ is not the characteristic
  function of $\{e \ | \ W_e\cap \ov A\neq \es\}$.

  Let $\{\phi_i\}_{i\in\omega}$ be a listing of all partial computable
  functions on two inputs.  Then we must meet P$_h$ for all
  $h=\phi_i$.  We find it easier to refer to $h$ instead of $i$ in
  this situation.  However, when we say $h<j$, please note that if we
  are using $h$ to stand in for $\phi_i$, then we mean that $i<j$.



  We will try to meet P$_h$ for each $\alpha$ of length $i$ for
  $h=\phi_i$ on our tree of strategies, and for each $e\in \omega$, so
  we will label our requirements P$_{h,e}^\alpha$.  To meet
  P$_h^\alpha$, for each $e$, if P$_{h,e}^\alpha$ has been reset $n$
  times, then we build $S_{h,e}^\alpha=W_{r(\alpha, e, n)}$, where $r$
  is known in advance by the Recursion Theorem with Parameters.
  If $h$ guesses that $S_{h,e}^\alpha\cap \ov A$ is empty, we add an
  element, called a {\em witness}, to $S_{h,e}^\alpha$.  If $h$
  changes its guess, then we want to enumerate the witness into $A$.
  When P$_{h,e}^\alpha$ gets reset, its witnesses become inactive.  We
  will meet the requirement for some $e$ and $\alpha$, so that
  $W_{r(\alpha,e,n)}\cap \ov A$ is nonempty if and only if
  $\lim_s h(r(\alpha,e,n),s)\neq 1$.

  Define $Y_h=\{x \ | \ x$ is ever a witness for some
  P$_{h,e}^\alpha$, for any $e$ and $\alpha\}$ and

  $Z_h=\{x\ | \ x$ is a witness in $Y_h$ that becomes inactive but
  does not enter $A\}.$

  Let $Y^j$ be defined as $\displaystyle \bigcup_{h\leq j} Y_h$ and
  $Z^j$ be defined as $\displaystyle \bigcup_{h\leq j} Z_h$.  Thus
  $Y^j$ is the set of witnesses for nodes of length at most $j$ and
  $Z^j$ is the subset of $Y^j$ of witnesses that become inactive.

  For $A$ to have semilow$_{1.5}$ complement, we need to ensure:

  L: $\{ j\ | \ W_j\cap \ov A $ infinite$\}\leq_m $ Inf.

  To do this, we will show that $W_j\cap \ov A$ is infinite if and
  only if either $W_j\cap Z^j$ is infinite or $W_j-(Y^j\cup A)$ has
  infinitely many expansionary stages, where a stage is expansionary
  if $(W_j-(Y^j\cup A))[s]$ is larger than it has been at any previous
  stage.  Since $W_j\cap Z^j$ and the set of expansionary stages of
  $W_j-(Y^j\cup A)$ are both \ce sets, this will give us the desired
  $m$-reduction.

  As mentioned, we will work on a binary tree of strategies.  Let
  $\alpha$ be a length $i$ string on the tree.  For each $j<i$,
  $\alpha(j)=0$ indicates a guess that $W_j-(Y^j\cup A)$ has
  infinitely many expansionary stages and $\alpha(j)=1$ is a guess
  that it has finitely many.  P$_h^\alpha$ will work to satisfy P$_h$
  while assuming that $\alpha$ is correctly guessing the true outcome.

  The tree allows us to ensure that if $W_j-(Y^j\cup A)$ has
  infinitely many expansionary stages, it is in fact an infinite set.
  The nodes guessing that there are infinitely many expansionary
  stages are not allowed to add witnesses when L$_j$ is injured.
  Instead, they wait until a new expansionary stage is reached.  This
  way, they do not create witnesses that may later enter
  $W_j-(Y^j\cup A)$, only to be forced into $A$.  In addition, nodes
  guessing that $W_j-(Y^j\cup A)$ has finitely many expansionary
  stages are reset each time $W_j-(Y^j\cup A)$ reaches a new
  expansionary stage, meaning that any of their current witnesses will
  never enter $A$.  The combination of these two actions guarantees
  that if $W_j-(Y^j\cup A)$ has infinitely many expansionary stages,
  it has infinitely many elements that never enter $A$ and so it is
  infinite.

  In addition, we need that $A\leT D$, which we achieve by a nonlow
  permitting argument using a technique of Downey, Jockusch, and
  Schupp \cite{MR3125901}.  An element only becomes a witness for
  P$_{h,e}^\alpha$ when $\Phi^D_e(e)[s]\converges$.  If this computation
  later converges, such a witness gets enumerated into $A$.  For each
  $\alpha$, we build a computable function on two inputs,
  $\ell_\alpha(e,s)$.  If P$_{h,e}^\alpha$ wishes to enumerate a
  witness into $A$ at stage $s$, we set $\ell_\alpha(e,s)=1$.  
  Otherwise $\ell_\alpha(e,s)=0$.  If the computation
  $\Phi^D_e(e)[s]$ is injured, then $D$ permits the witness to enter
  $A$ and so the witness enters $A$ as desired.
  We will only enumerate witnesses when permitted, so that
  we will achieve $A\leT D$.  Thus, we must ensure that $D$ permits
  frequently enough.  Since $D$ is not low and $\ell_\alpha$ is computable,
  we know that $K^D(e)$ is not the limit of $\ell_\alpha(e,s)$ as $s$ 
  approaches infinity, else $K^D$ would be limit computable and thus
  $D$ would be low.  Therefore, there
  must be some $e$ such that $K^D(e)\neq \lim_s \ell_\alpha(e,s)$.  
  When P$_{h,e}^\alpha$ wishes to enumerate an element, 
  $\ell_\al(e,s)=1$, so $\Phi^D_e(e)$ must diverge.  When the current
  computation is injured by a new element entering $D$ below the use, P$_{h,e}^\al$
  has permission to enumerate the element into $A$.  Thus, for each $\al$ on the true path, there is some
  $e$ such that $D$ will give permission for P$_{h,e}^\al$ to enumerate elements into $A$ as often as needed.

  \subsection{Construction}

  For each $k$, we define a finite set $F_k$.  Each of these sets will
  be contained within $\omega^{[0]}$, where $\omega^{[n]}$ is the
  $n$th ``row'' of $\omega$, the set of all elements of the form
  $\<x,n\>$ using the standard pairing function.  We let $F_0$ be the
  singleton set $\{\<0,0\>\}$, and for $k>0$, we let $F_k$ be the
  least $2k^2$ elements in $\omega^{[0]}$ that have not appeared in
  any $F_{k'}$ for $k'<k$.

  We build an approximation $\delta_s$ to the true path.  We call $s$
  an $\alpha$-stage if $\alpha$ is a substring of $\delta_s$.  Define
  $\delta_0=\lambda$, the empty string.  Let $\delta_s$ be the length
  $s$ string defined by: $\delta_s(j)=0$ if $W_j-(Y^j\cup A)$ has
  reached a new expansionary stage since the last
  $\delta_s\res j$-stage, and $1$ otherwise.

  Let $\ell_\alpha(e,0)=0$ and $\ell_\alpha(e,s+1)=\ell_\alpha(e,s)$
  unless otherwise specified.
  We begin with all L$_j$ allowing addition of new witnesses for every
  P$_h^\alpha$.  During the construction, in order to prevent further injury
  to itself, L$_j$ may disallow some P$_h^\alpha$ from adding new witnesses.
  However, we will show that if $\alpha$ is on the true path, the construction
  ensures that P$_h^\alpha$ will be allowed to add witnesses infinitely often.

  Nodes on the tree may be put in 1-1 correspondence with
  $\mathbb N-\{0\}$ via a function $f:2^{<\omega}\ra \mathbb N-\{0\}$
  .  Using such a correspondence, let
  $\langle \alpha, e\rangle :=\langle f(\alpha), e\rangle$.  In Step
  3, witnesses come from $\omega^{[\langle \alpha, e\rangle]}$.

  {\em Stage $s=0$}: $\ell_\al(e,0)=0$ and all sets $S_{h,e}^\al$
  begin empty for all $\al$, $h$, and $e$.

  {\em Stage $s+1$}:

  {\bf Step 1} (new expansionary stage): For each $j\in \omega$, if
  $W_j-(Y^j\cup A)$ grows to a size we haven't seen before, i.e. it
  reaches a new expansionary stage, then do the following for each
  $\alpha$ of length greater than $j$: 
  If $\alpha$ guesses that $W_j-(Y^j\cup A)$ has infinitely many
  expansionary stages (i.e. $\alpha(j)=0$), then L$_j$ now allows
  addition of new witnesses for P$_h^\alpha$.  If $\alpha$ guesses
  that $W_j-(Y^j\cup A)$ has finitely many expansionary stages
  ($\alpha(j)=1$), we reset P$_{h,e}^\alpha$ for all $e$, making
  $\ell_\alpha(e,s)=0$ and declaring P$_{h,e}^\alpha$ and any
  witnesses for P$^\alpha_{h,e}$ inactive.  When reset, we will start
  a new $S_{h,e}^\alpha$ that begins empty.
  This will ensure in Lemma \ref{semilowonefive} that nodes guessing incorrectly that $W_j-(Y^j\cup A)$ has finitely many expansionary stages will not prevent $W_j\cap \ov A$ from being infinite by putting too many elements into $A$.

  {\bf Step 2} (allowing some disallowed witnesses): We say a node of
  length $j$ {\em is allowed} new witnesses if no $j'<j$ disallows
  P$_h^\alpha$ from adding witnesses.  For each $j$, $N$, and $\beta$,
  where $\beta$ has length $j$, if L$_j$ currently disallows for $N$
  (as defined in Step 5) and if $\beta$ is allowed new witnesses, then
  set L$_j$ to allow new witnesses for all $\alpha$ unless Step 2 has
  acted at some previous stage for these particular $j$, $N$, and
  $\beta$.
  Do this for all possible triples, of which there are only finitely
  many.  This step guarantees that nodes on the true path will be
  allowed to act in Step 3 infinitely often.

  {\bf Step 3} (adding witnesses): We act for each $\alpha$
  of length at most $s$
  such that P$_h^\alpha$
  is allowed to add witnesses by each L$_j$ with $j<s$ such that
  $\alpha(j)=0$, 
  and such that either $s$
  is an $\alpha$-stage
  or $\alpha$
  is to the left of $\delta_s$
  and if $s_\alpha$
  was the last $\alpha$-stage,
  P$_h^\alpha$
  has not been allowed to act at any stage $t$,
  $s_\alpha\leq
  t<s$.  In other words, we act on the approximation to the true path
  $\delta_s$
  and for any nodes to the left of $\delta_s$
  that have not been allowed to act since before they were last on the
  approximation to the true path.  We act in length-lexicographic
  order. For each such $\alpha$,
  we ask if there is any $e<s$ such that

  \begin{itemize}
  \item P$_{h,e}^\alpha$ is inactive,
  \item $\Phi_e^D(e)[s]\converges$,
  \item $h(r(\alpha,e,n))=0$
    for $n$ the number of times P$_{h,e}^\alpha$ has been reset, and
  \item for all $e'<
    e$, $\Phi_{e'}^D(e')[s]\converges\iff
    \ell_\alpha(e',s)=1$ or
    $\Phi_{e'}^D(e')[s]\converges$
    and has changed computations at least $e$
    times since the last time that P$_{h,e}^\alpha$ added a witness.
  \end{itemize}

  If so, then for each such $e$, let $x>s$,
  $x\in \omega^{[\langle \alpha, e\rangle]}$, be a fresh witness,
  larger than any previously chosen witness.  Enumerate
  $x\in S_{h,e}^\alpha$.  Declare $x$ to have use $s$, which we note
  is larger than the use of
  $\Phi_e^D(e)[s]$.  
  We call both P$_{h,e}^\alpha$ and $x$ {\em active}.

  {\bf Step 4} (changing $\ell_\alpha$ when $h$ changes its guess): If
  P$_{h,e}^\alpha$ is active, $\ell_\alpha(e,s)=0$, and
  $h(r(\alpha, e, n),s+1)=1$ where $n$ is the number of times
  P$_{h,e}^\alpha$ has been reset, then set $\ell_\alpha(e,s+1)=1$.
  Reset P$_{h,e'}^\alpha$ for all $e'>e$ by making P$_{h,e'}^\alpha$
  and all elements in $S_{h,e'}^\alpha$ inactive and starting new
  empty sets $S_{h,e'}^\alpha$, as well as setting
  $\ell_\alpha(e',s+1)=0$. Perform this step for each $\alpha$ and for
  the least $e$ such that it applies.

  {\bf Step 5} (enumerating witnesses into A and disallowing new
  witnesses): If $x\in S_{h,e}^\alpha$ is active with use $u$ and
  $D_{s+1}\res u\neq D_{s}\res u$, then put $x$ into $A_{s+1}$ and set
  $\ell_\alpha(e,s+1)=0$.  
  Declare P$_{h,e}^\alpha$ and $x$ inactive.

  If by such an enumeration, we cause an element of
  $(W_j-(Y^j\cup A))[s]$ to enter $A$, then we say L$_j$ does not {\em
    allow} P$_h^\alpha$ to add witnesses for any $\alpha$ such that
  $\alpha(j)=0$.  In addition, if the size of $(W_j-(Y^j\cup A))$ has
  reached the number $N$, we say that L$_j$ {\em disallows for
    $N$}.

%
%

  {\bf Step 6} (coding D): If $k\in D_{s+1}-D_s$, enumerate one
  element from $F_k$ into $A_{s+1}$ such that for each $M>0$, if
  $j+M=k$, then the element is not one of the
  least 
  $M$ elements in $(W_j-(Y^j\cup
  A))[s]$.  

  \subsection{Verification}

  \begin{lemma}\label{semilowonefive} For each $j\in \omega$, if the
    set of expansionary stages of $W_j-(Y^j\cup A)$ is infinite, then
    $W_j\cap \ov A$ is infinite.
  \end{lemma}

  \begin{proof}

    Assume the set of expansionary stages of $W_j-(Y^j\cup A)$ is
    infinite.  We will show that for each $M$, $W_j-(Y^j\cup A)$ has
    at least $M$ elements.

    No element in $(W_j-(Y^j\cup A))[s]$ ever enters $Y^j$ since new
    witnesses are chosen to be larger than $s$, and elements of
    $(W_j-(Y^j\cup A))[s]$ are smaller than $s$.%

    Induct on $M$.  Assume true for
    $M-1$.  

    For $k\leq
    j$, only finitely many elements are ever put into $A$ by
    C$_k$.
    For $k>j$,
    C$_k$
    can only bring the size of $W_j-(Y^j\cup
    A)$ below $M$ if $M+j>k$, which happens finitely often.  Let
    $s_0$
    be a stage by which the least $M-1$
    elements in $(W_j-(Y^j\cup
    A))[s_0]$ never enter $A$ and after which no
    C$_k$
    enumerates any of the least $M$
    elements of $(W_j-(Y^j\cup
    A))[s]$ into $A$.  Let
    $s_1>s_0$
    be a stage after which Step 2 never acts for ($j,$
    $M-1$,
    $\beta$)
    for any $\beta$
    of length $j$.
    We may assume that at every stage $s>s_1$,
    the $M$th
    least element of $(W_j-(Y^j\cup
    A))[s]$ is a witness for some
    P$_{h,e}^\alpha$,
    else it would never enter $A$
    and we would be guaranteed to have at least $M$
    elements in $W_j-(Y^j\cup
    A)$, as desired.  Note that the length of
    $\al$ is greater than $j$ since the witnesses are not in $Y^j$.

    Let $s_2>s_1$
    be a stage such the $M$th
    least element in $(W_j-(Y^j\cup
    A))[s_2]$, called
    $x_M$,
    is the least that will ever be in the $M$th
    position of any $(W_j-(Y^j\cup
    A))[s]$ for $s\geq s_2$.  Now, if $x_M$ never enters
    $A$
    by Step 5, we are done, so assume $x_M$
    enters $A$.
    When $x_M$
    enters $A$,
    all larger witnesses also enter $A$,
    and since the next element in the $M$th
    position is a witness, and it cannot be less than $x_M$
    due to minimality, then the next element that enters the
    $M$th
    position has yet to become a witness.  When $x_M$
    enters $A$,
    L$_j$
    does not allow any P$_h^\alpha$
    to add witnesses for any $\alpha$
    such that $\alpha(j)=0$.
    Step 2 has finished acting for $j$
    and $M-1$,
    so it will not cause L$_j$
    to later allow.  The next time L$_j$
    allows will be by Step 1, which means $(W_j-(Y^j\cup
    A))$ has already reached a new expansionary stage.  Thus, at the
    first expansionary stage after $x_M$
    enters $A$,
    the only witness that could be in the $M$th
    position must be a witness for an $\alpha$
    with $\alpha(j)=1$.
    At that new expansionary stage, Step 1 inactivates the witnesses
    for $\alpha$.
    Thus the element in the $M$th
    position can never enter $A$,
    so $(W_j-(Y^j\cup A))$ has at least $M$
    elements. 

  \end{proof}

  We say that $\alpha$ is on the {\em true path} if $\alpha(j)=0$ if
  and only if $W_j-(Y^j\cap \ov A)$ has infinitely many expansionary
  stages during the construction.


  \begin{lemma}\label{notsemi} Assume $\alpha$ is on the true path.
    Let $e$ be the least such that
    $\displaystyle \lim_s \ell_\alpha(e,s)\neq K^D(e)$.  Then
    P$_{h,e}^\alpha$ is reset finitely often ($n$ times) and
    $\displaystyle \lim_s h(r(\alpha, e,n), s)=0\iff W_{r(\alpha, e,
      n)}\cap \ov A $ is nonempty.  Thus, $\ov A$ is not semilow.
  \end{lemma}

  \begin{proof} P$_{h,e}^\alpha$ can only be reset in two ways: In
    Step 1 by $W_j-(Y^j\cup A)$ growing while $\alpha$ is guessing it
    has finitely many expansionary stages, which can only happen
    finitely often for the correct $\alpha$, and by Step 4 acting for
    P$_{h,e'}^\alpha$, for $e'<e$.  Such a P$_{h,e'}^\alpha$ only
    resets P$_{h,e}^\alpha$ when $\ell_\alpha(e',s)$ changes to 1 from
    0.  This can only happen finitely often since
    $\displaystyle \lim_s \ell_\alpha(e',s)$ must exist for $e'<e$.
    Thus, P$_{h,e}^\alpha$ is reset $n$ times for some finite $n$.


    {\em Case 1}: Suppose $S^\al_{h,e}\cap \ov A$ is nonempty, for
    $S^\al_{h,e}=W_{r(\al, e, n)}$.

    Then there is some $x\in S_{h,e}^\al\cap \ov A$ that was
    enumerated into $S_{h,e}^\alpha$ by Step 3 at stage $s_0+1$ when
    $\Phi_e^D(e)[s_0]\converges$ and $\ell_\alpha(e,s_0)=0$.  We want
    to show that $\lim_s h(r(\alpha,e,n),s)=0$.

    Assume $h(r(\alpha,e,n),t)$ equals 1 at some stage $t>s_0$, so
    $\ell_\al(e,t)=1$.  Since $x\in \ov A$,
    $\Phi_e^D(e)[s_0]=\Phi_e^D(e)\converges$ and both computations are
    the
    same. 
    Since $e$ was chosen so that $\lim_s \ell_\alpha(e,s)\neq K^D(e)$,
    we know that $\ell_\alpha(e,s)$ must change back to 0, which can
    only happen in Step 5.  However, this can't happen since
    $\Phi_e^D(e)[s_0]\converges$ and never changes below the use.
    This gives a contradiction.  Thus, if $S_{h,e}^\alpha\cap \ov A$
    is nonempty, then $\displaystyle \lim_s h(r(\alpha, e, n),s)=0$.

    {\em Case 2:} Suppose $S_{h,e}^\alpha\cap \ov A$ is empty.

    First, note that if $S_{h,e}^\al$ has infinitely many witnesses,
    all of which enter $A$, then Step 5 acts infinitely often, meaning
    that $\Phi_e^D(e)[s]$ changes infinitely often, so $\Phi_e^D(e)$
    diverges, and thus $\lim_s \ell_\al(e,s)\neq 0$.  Since
    $\ell_\al(e,s)=1$ infinitely often, we must have that
    $h(r(\al,e,n),s)$ switches from $0$ to $1$ infinitely often, as it
    must be $0$ when each witness is appointed and $1$ when
    $\ell_\al(e,s)$ changes to 1.  Thus $\lim_s h(r(\al,e,n,s)$
    doesn't exist.

    We may thus assume that after a finite stage,
    $S_{h,e}^\alpha\cap \ov A$ is empty and never gets more elements.
    We may also assume that $\lim_s h(r(\alpha, e, n),s)=0$, else we
    have achieved our goal.  We now show that P$_{h,e}^\alpha$ will
    eventually be able to act by adding a new witness in Step 3,
    contradicting that $S_{h,e}^\alpha$ never gets a new element.  We
    check that all the conditions of Step 3 will eventually be met.

    For the first bullet point, we know that P$_{h,e}^\alpha$ will be
    inactive for almost all $s$ since it is only active when
    $S_{h,e}^\alpha\cap \ov A$ is nonempty.  Since P$_{h,e}^\alpha$ is
    inactive, we know that $\ell_\alpha(e,s)=0$ at almost all stages,
    as it only can be 1 when P$_{h,e}^\alpha$ is active.  Now, we also
    know that $\lim_s \ell_\alpha(e,s)\neq K^D(e)$, so
    $\Phi^D_e(e)\converges$.  This gives us the second bullet point
    for almost all $s$.  By the assumption in the previous paragraph,
    for almost all $s$, $h(r(\alpha, e, n),s)=0$, so we meet the third
    bullet point.

    To check that the fourth bullet point is met, note that for all
    $e'<e$, either $K^D(e')=1=\lim_s \ell_\alpha(e',s)$, in which case
    for almost all $s$, $e'$ does not prevent $e$ from acting, or
    $K^D(e')=0=\lim_s \ell_\alpha(e',s)$.  In the latter case, there
    are two possibilities.  If for almost all $s$,
    $\Phi^D_{e'}(e')[s]\diverges$, then for almost all $s$, $e'$ does
    not prevent $e$ from acting.  Otherwise, there are infinitely many
    stages such that $\Phi_{e'}^D(e')[s]$ converges, but since
    $\Phi_{e'}^D(e')\diverges$, the computation changes infinitely
    often, so eventually the computation will have changed at least
    $e$ times.  Thus, $e'$ only prevents $e$ from adding a new witness
    for finitely many stages.

    We have shown that for almost all $s$, the conditions in all the
    bullet points are met.  We still must show that $\alpha$ is {\em
      allowed} to add a witness at a stage where it is able to
    act.  

    {\em Claim}: For every $\alpha$ on the true path, $\alpha$ is
    allowed to add witnesses infinitely often.

    {\em Proof of Claim}: Induct on the length of $\alpha$.  Assume
    true for the predecessor of $\alpha$, $\alpha^-$.  Let $\alpha^-$
    have length $j$.  If $\alpha(j)=1$, then $\alpha$ is allowed to
    act at every stage that $\alpha^-$ is allowed to act.  If
    $\alpha(j)=0$, we need to show that infinitely often, when
    $\alpha^-$ is allowed to act, L$_j$ allows action as well.
    Suppose not.  Then there is some stage $s_0$ such that $\alpha$ is
    not allowed to act for any $s\geq s_0$.  %
    Each time $\alpha^-$ is allowed to act after stage $s_0$, L$_j$
    disallows $\alpha$ from adding new witnesses.  This means that
    Step 2 does not apply for $j$ at any of these
    stages.  
    We know that by the proof of Lemma \ref{semilowonefive}, L$_j$
    disallows for each $N$ at most finitely often.  Since $\alpha$ is
    on the true path, $W_j-(Y^j\cup A)$ has infinitely many
    expansionary stages, so L$_j$ eventually disallows only for $N$'s
    that Step 2 never acted for before stage $s_0$.  Thus eventually
    Step 2 will apply to $j$, $N$, and $\alpha^-$, so $\alpha$ will be
    allowed to add new elements, which gives a contradiction, proving
    the claim.

    Now wait until a stage $s_0$ such that the approximation to the
    true path never goes to the left of $\alpha$ and the bullet points
    do not prevent P$_{h,e}^\alpha$ from acting.  Let $s_1>s_0$ be the
    next $\alpha$-stage, and $s_2\geq s_1$ the next stage that
    $\alpha$ is allowed to add witnesses.  Then $s_2$ is either an
    $\alpha$-stage or $\delta_{s_2}$ is to the right of $\alpha$, so
    we may perform Step 3 for
    P$_h^\alpha$. 




    At infinitely many stages, Step 3 will be able to act by adding a
    witness.  Thus, it will act, contradicting that no new witnesses
    are added.  Therefore, if no new witnesses are ever added, then we
    must have that $\lim_s h(r(\alpha, e, n),s)$ does not equal $0$,
    satisfying the lemma.

%
  \end{proof}

  \begin{lemma}\label{subseteqstar} For each partial computable
    $\phi_i=h$, $Y_h\cap \ov A\subseteq^* Z_h$.
  \end{lemma}

  \begin{proof} Proving this lemma amounts to showing that there are
    only finitely many elements in $Y_h\cap \ov A$ that remain active.
    For each $e$ and each $\alpha$, there can only be one active
    element at a time.  Let $e$ be such that P$_{h,e}^\alpha$ is met,
    as in Lemma \ref{notsemi}.  Either $e$ resets all $e'>e$
    infinitely often or $\lim_s \ell_\alpha(e,s)$ exists.

    If $\lim_s \ell_\alpha(e,s)=0$, then $\Phi_e^D(e)\converges$.
    Then $\Phi_e^D(e)[s]$ changes computations some finite number of
    times, so by the final bullet of Step 3, only finitely many $e'>e$
    will ever be able to add witnesses, and will only be able to do so
    finitely
    often.  
    If $\lim_s \ell_\alpha(e,s)=1$, then $\Phi_e^D(e)\diverges$, but
    this is not possible since $\ell_\alpha(e,s)=0$ if
    $\Phi_e^D(e)[s]\diverges$.

    Thus, for a fixed $\alpha$ and $h$, either almost all
    P$_{h,e}^\alpha$ are reset infinitely often or only finitely many
    ever become active.  Since there are only finitely many $\alpha$
    of length $i$, $h=\phi_i$, we conclude that
    $Y_h\cap \ov A\subseteq^* Z_h$.
  \end{proof}

  \begin{lemma}\label{semiparttwo} For each $j\in \omega$,
    $W_j\cap \ov A$ is infinite if and only if either the set of
    expansionary stages of $W_j-(Y^j\cup A)$ is infinite or
    $W_j\cap Z^j$ is infinite.  Thus, $\ov A$ is semilow$_{1.5}$.
  \end{lemma}

\begin{proof}
  This Lemma gives us that $\ov A$ is semilow$_{1.5}$ because to tell
  if $W_j\cap \ov A$ is infinte, we can ask Inf if the \ce set that is
  the union of $W_j\cap Z^j$ and the set of expansionary stages of
  $W_j-(Y^j\cup A)$ is infinite.  This gives us an $m$-reduction from
  $W_j\cap \ov A$ to Inf.

  First, note that $Z^j\subseteq \ov A$, so if $W_j\cap Z^j$ is
  infinite, so is $W_j\cap \ov A$.  By Lemma \ref{semilowonefive}, we
  know that if the set of expansionary stages of $W_j-(Y^j\cup A)$ is
  infinite, then $W_j\cap \ov A$ is infinite.

  Assume $W_j\cap \ov A$ is infinite.  Suppose $W_j-(Y^j\cup A)$ has
  finitely many expansionary stages.  Then $W_j-(Y^j\cup A)$ is
  finite.  Thus, $W_j\cap \ov A\subseteq^* W_j\cap Y^j$.  By Lemma
  \ref{subseteqstar}, $Y_h\cap \ov A\subseteq^* Z_h$, so
  $W_j\cap \ov A\subseteq^* W_j\cap Z^j$.  Thus $W_j\cap Z^j$ is
  infinite.
\end{proof}

\begin{lemma} $A\leT D$.
\end{lemma}

\begin{proof} To determine if $k\in D$, ask if any elements of the set
  $F_k$ are in $A$.  The only way any of the elements can enter $A$ is
  if $k\in D$.  It is always possible to enumerate one of the elements
  into $A$ because we must keep out at most
  $\displaystyle\sum_{M=1}^k M=k(k+1)/2$ elements, and
  $\frac{(k+1)k}2=\frac{k^2}2+\frac k2<2k^2$, which is the size of
  $F_k$ for $k>0$.  For $k=0$, we do not need to keep out any
  elements, and $F_0$ is nonempty.
\end{proof}

\begin{lemma} $D\leT A$.
\end{lemma}

\begin{proof}
  $\omega^{[0]}\cap A\leT D$ since no element from $\omega^{[0]}$
  enters $A$ unless it is in $F_k$ and $k$ enters $D$.

  For $\omega-\omega^{[0]}$, $x\in A$ if and only if $x$ becomes a
  witness and then enters $A$.  The set of witnesses is computable
  because if $x$ is not a witness by stage $x$, it will never become a
  witness. If $x$ is a witness, we ask what its use $u$ is as a
  witness.  Let $s_x$ be such that $D\res u=D_{s_x}\res u$.  Then
  $x\in A$ if and only if $x\in A_{s_x}$.
\end{proof}

\end{proof}


\section{Semilow$_{2}$, O.S.P, not semilow$_{1.5}$ in all nonlow
  degrees}

\label{semilow2OSP}

As mentioned previously, a set $B$ is {\em semilow$_2$} if $\{e\ | \ W_e\cap B$ is
infinite$\}\leT \0''$.  It follows immediately that if $B$ is
semilow$_{1.5}$, then $B$ is semilow$_2$.

A set $A$ has the {\em outer splitting property} if there exist total
computable functions $f$ and $g$ such that, for each $i\in \omega$,
\begin{itemize}
\item[(a)] $W_i=W_{f(i)}\sqcup W_{g(i)}$,
\item[(b)] $W_{f(i)}\cap \ov A$ is finite, and
\item[(c)] if $W_i\cap \ov A$ is infinite, then $W_{f(i)}\cap \ov A$
  is nonempty.
\end{itemize}

Maass \cite{Maass:83} showed that if a set $A$ has semilow$_{1.5}$ complement,
then $A$ has the outer splitting property.  Thus, the class of \ce sets with semilow$_{1.5}$ complement is contained in the intersection of
the sets with semilow$_2$ complement and the sets with the outer splitting property.  We show that in every nonlow \ce degree, this containment is strict.

\begin{theorem}\label{TwoOSPnotOneFive} For every nonlow \ce degree
  $\bf d$, there is a \ce set $A\in \bf d$ such that $A$ has the outer
  splitting property and semilow$_{2}$ complement, but does not have
  semilow$_{1.5}$ complement.
\end{theorem}

\begin{Cor} The nonlow \ce degrees are precisely the degrees of \ce sets that have the outer splitting property and semilow$_2$ complement but not semilow$_{1.5}$ complement.
\end{Cor}

The corollary follows immediately from the theorem, since every low \ce set has semilow$_{1.5}$ complement.

\begin{proof}

  Given a nonlow \ce set $D$, we will construct an $A\equivT D$ such
  that $\ov A$ is semilow$_{2}$ but not semilow$_{1.5}$ and such that
  $A$ satisfies the outer splitting property.

  \subsection{Requirements:}

  To ensure that $D\leT A$, we will code $D$ into $A$.  To begin, we
  construct a computable list of disjoint finite sets
  $F_k\subset \omega^{[0]}$, such that $F_0$ contains two elements and
  each other $F_k$ has $6k^2$ elements.

  C$_k$: $k\in D$ if and only if some element of $F_k$ is in $A$.

  To make $\ov A$ not semilow$_{1.5}$, we meet the following
  requirement for each total computable function $h$:

  P$_h$: $\{e \ | \ W_e \cap \ov A$
  infinite$\} \neq \{e \mid W_{h(e)} $ infinite$\}$.

  Let $\{\phi_i\}_{i\in\omega}$ be a listing of all partial computable
  functions on two inputs.  Then we must meet P$_h$ for all
  $h=\phi_i$.  As in Theorem \ref{onefivenotsemilow}, we find it
  easier to refer to $h$ instead of $i$ in this situation and when we
  say $h<j$, we mean that $i<j$.



  We will try to meet P$_h$ for each $\alpha$ of length $i$ for
  $h=\phi_i$ on our tree of strategies and each $e\in\omega$, so we
  will label our requirements P$_{h,e}^\alpha$.  To meet P$_h^\alpha$,
  for each $e$, if P$_{h,e}^\alpha$ has been reset $n$ times, then we
  build $S_{h,e}^\alpha=W_{r(\alpha, e,
    n)}$, 
  where $r$is  known in advance by the Recursion Theorem with
  Parameters.  We add elements, called witnesses, to $S_{h,e}^\alpha$.
  For some $e$, we will guarantee that if $W_{h(r(\alpha,e,n))}$ is
  infinite, then we dump the witnesses into $A$, so $S_{h,e}^\alpha\cap \ov A$
  is finite, and if $W_{h(r(\alpha,e,n))}$ is finite, $S_{h,e}^\alpha$
  gets infinitely many witnesses that never enter $A$.  Thus,
  $r(\alpha,e,n)$ will satisfy that $W_{r(\alpha,e,n)}\cap \ov A$ is
  infinite if and only if $W_{h(r(\alpha,e,n))}$ is
  finite. 

  Define

  $Y_{h,e}^\alpha=\{x\ | \ x$ is ever a witness for
  P$_{h,e}^\alpha\},$

  $Y_h=\{x \ | \ x$ is in $Y_{h,e}^\alpha$ for some $e$ and
  $\alpha\}$, and

  $Y^j=\displaystyle \bigcup_{h\leq j} Y_h$.
  
  Thus, $Y^j$ is the set of all $x$ that are ever witnesses for any $P_{h,e}^\al$ for any $h\leq j$.  

  To meet the outer splitting property, we build functions $f$ and $g$
  so that for each $i$, we satisfy requirement O$_i$:
  \begin{itemize}
  \item[(a)] $W_i=W_{f(i)}\sqcup W_{g(i)}$,
  \item[(b)] $W_{f(i)}\cap \ov A$ is finite, and
  \item[(c)] if $W_i\cap \ov A$ is infinite, then $W_{f(i)}\cap \ov A$
    is nonempty.
  \end{itemize}

  For $A$ to have semilow$_{2}$ complement, we need to ensure $\{ j\ | \ W_j\cap \ov A $ infinite$\}\leT \bf 0''$.

  To achieve this, we will meet for each $j\in\omega$:

  L$_j$: $W_j\cap \ov A $ is infinite if and only if either
  $W_j-(Y^j\cup A)$ has infinitely many expansionary stages or
  $W_j\cap Y^j\cap \ov A$ is infinite.

  We will also show that $\bf 0''$ is able to determine if
  $W_j\cap Y^j\cap \ov A$ is infinite.

  We use a binary tree of strategies.  Let $\alpha$ be a length $i$
  string on the tree.  For each $j<i$, $\alpha(j)=0$ indicates a guess
  that $W_j-(Y^j\cup A)$ has infinitely many expansionary stages and
  $\alpha(j)=1$ is a guess that it has finitely many.  P$_h^\alpha$
  will work to satisfy P$_h$ while assuming that $\alpha$ is correctly
  guessing the true outcome.

  The tree serves two purposes.  As in Theorem
  \ref{onefivenotsemilow}, it allows us to ensure that if
  $W_j-(Y^j\cup A)$ has infinitely many expansionary stages, it is in
  fact an infinite set.  The nodes guessing that there are infinitely
  many expansionary stages are not allowed to add witnesses when L$_j$
  is injured.  Instead, they wait until a new expansionary stage is
  reached.  This way, they do not create witnesses that may later
  enter $W_j-(Y^j\cup A)$, only to be forced into $A$.  In addition,
  nodes guessing that $W_j-(Y^j\cup A)$ has finitely many expansionary
  stages are reset each time $W_j-(Y^j\cup A)$ reaches a new
  expansionary stage, meaning that any of their current witnesses will
  never enter $A$.  The combination of these two actions guarantees
  that if $W_j-(Y^j\cup A)$ has infinitely many expansionary stages,
  it has infinitely many elements that never enter $A$ and so it is
  infinite.

  The other purpose of the tree is to allow $\0''$ to compute whether
  $W_j\cap Y^j\cap \ov A$ is infinite.  In Lemma \ref{zerodouble}, we
  analyze which elements of $Y^j$ stay in $\ov A$ forever by examining
  what happens to witnesses for P$_h^\alpha$ when $\alpha$ has various
  relationships to the true path.  Since $\0''$ can determine the true
  path, we can use $\0''$ to find the index for a \ce set that is
  finitely different from $Y^j\cap \ov A$.  Thus $\0''$ can determine
  whether $W_j\cap Y^j\cap \ov A$ is infinite.

  In addition, we need that $A\leT D$, which we achieve by a nonlow
  permitting argument, as in Theorem \ref{onefivenotsemilow}.

  \subsection{Construction}

  We build an approximation $\delta_s$ to the true path.  We call $s$
  an $\alpha$-stage if $\alpha$ is a substring of $\delta_s$.  Define
  $\delta_0=\lambda$, the empty string.  Let $\delta_s$ be the length
  $s$ string defined by: $\delta_s(j)=0$ if $W_j-(Y^j\cup A)$ has
  reached a new expansionary stage since the last
  $\delta_s\res j$-stage, and $1$ otherwise.

  Let $\ell_\alpha(e,0)=0$ and $\ell_\alpha(e,s+1)=\ell_\alpha(e,s)$
  unless otherwise specified.

  We begin with all L$_j$ allowing addition of new witnesses for every
  P$_h^\alpha$.  In Step 5, if L$_j$ is injured, then it will disallow
  new witnesses for P$_h^\alpha$, where $\alpha$ guesses the infinite
  outcome for L$_j$.  L$_j$ will again allow new witnesses if
  $W_j-(Y^j\cup A)$ reaches a new expansionary stage (Step 1).  Step 2
  also has the ability to make L$_j$ allow again, and serves the
  purpose of ensuring that if $W_j-(Y^j\cup A)$ has infinitely many
  expansionary stages, it allows each P$_h^\alpha$ to add witnesses
  infinitely often.

  Nodes on the tree may be put in 1-1 correspondence with
  $\mathbb N-\{0\}$ via a function $f:2^{\in\omega}\ra \mathbb N-\{0\}$
  .  Using such a correspondence, let
  $\langle \alpha, e\rangle :=\langle f(\alpha), e\rangle$.  In Step
  3, witnesses come from $\omega^{[\langle \alpha, e\rangle]}$.
  
  Note that in the following, Steps 1 and 2 are identical to those in Theorem \ref{onefivenotsemilow}.

  {\em Stage $s+1$}:

%

{\bf Step 1} (new expansionary stage): For each $j\in \omega$, if
  $W_j-(Y^j\cup A)$ grows to a size we haven't seen before, i.e. it
  reaches a new expansionary stage, then do the following for each
  $\alpha$ of length greater than $j$: If $\alpha$ guesses that
  $W_j-(Y^j\cup A)$ has infinitely many expansionary stages
  (i.e. $\alpha(j)=0$), then L$_j$ now allows addition of new
  witnesses for P$_h^\alpha$.  If $\alpha$ guesses that
  $W_j-(Y^j\cup A)$ has finitely many expansionary stages
  ($\alpha(j)=1$), we reset P$_{h,e}^\alpha$ for all $e$, making
  $\ell_\alpha(e,s)=0$ and declaring P$_{h,e}^\alpha$ and any
  witnesses for P$^\alpha_{h,e}$ inactive.  When reset, we will start
  a new $S_{h,e}^\alpha$ that begins empty.

  {\bf Step 2} (allowing some disallowed witnesses): We say a node of
  length $j$ {\em is allowed} new witnesses if no $j'<j$ disallows
  P$_h^\alpha$ from adding witnesses.  For each $j$, $N$, and $\beta$,
  where $\beta$ has length $j$, if L$_j$ currently disallows for $N$
  (as defined in Step 5) and if $\beta$ is allowed new witnesses, then
  set L$_j$ to allow new witnesses for all $\alpha$ unless Step 2 has
  acted at some previous stage for these particular $j$, $N$, and
  $\beta$.
  Do this for all possible triples, of which there are only finitely
  many.  This step guarantees that nodes on the true path will be
  allowed to act in Step 3 infinitely often.

  {\bf Step 3} (adding witnesses): We act for each $\alpha$ of length
  at most $s$ such that P$_h^\alpha$ is allowed to add witnesses by
  each L$_j$ with $j<s$ such that $\alpha(j)=0$, 
  and such that either $s$ is an $\alpha$-stage or $\alpha$ is to the
  left of $\delta_s$ and if $s_\alpha$ was the last $\alpha$-stage,
  P$_h^\alpha$ has not been allowed to act at any stage $t$,
  $s_\alpha\leq t<s$.  In other words, we act on the approximation to
  the true path $\delta_s$ and for any nodes to the left of $\delta_s$
  that have not been allowed to act since before they were last on the
  approximation to the true path.  We act in length-lexicographic
  order. For each such $\alpha$, we ask if there is any $e<s$ such
  that

  \begin{itemize}
  \item $\ell_\alpha(e,s)=0$,
  \item $\Phi_e^D(e)[s]\converges$, and
  \item for all $e'< e$,
    $\Phi_{e'}^D(e')[s]\converges\iff \ell_\alpha(e',s)=1$ or
    $\Phi_{e'}^D(e')[s]\converges$ and has changed computations at
    least $e$ times since the last time that P$_{h,e}^\alpha$ added a
    witness.  
  \end{itemize}

  If so, then for each such $e$, let $x>s$,
  $x\in \omega^{\langle \alpha, e\rangle}$, be a fresh witness, larger
  than any previously chosen witness.  Enumerate
  $x\in S_{h,e}^\alpha$.  Declare $x$ to have use $\phi_e^D(e)[s]$,
  which is defined to be the use of $\Phi_e^D(e)[s]$.  We call the
  witness $x$ {\em active}.

  {\bf Step 4} (changing $\ell_\alpha$ when $W_{h(r(\alpha,e,n)}$
  grows): Suppose 
  $\ell_\alpha(e,s)=0$, $S_{h,e}^\alpha\cap \ov A$ contains at least
  one witness, and $|W_{h(r(\alpha,e,n),s+1}|>|W_{h(r(\alpha,e,n),s}|$
  where $n$ is the number of times P$_{h,e}^\alpha$ has been reset.
  Then set $\ell_\alpha(e,s+1)=1$ and reset P$_{h,e'}^\alpha$ for all
  $e'>e$
  by 
  starting new empty sets $S_{h,e'}^\alpha$, as well as setting
  $\ell_\alpha(e',s+1)=0$ and declaring any witnesses for
  P$_{h,e'}^\alpha$ inactive.  Perform this step for each $\alpha$ and
  for the least $e$ such that it applies.

  {\bf Step 5} (enumerating witnesses into A and disallowing new
  witnesses): If $x\in S_{h,e}^\alpha$ is active with use $u$ and
  $D_{s+1}\res u\neq D_{s}\res u$, then put $x$ into $A_{s+1}$ and set
  $\ell_\alpha(e,s+1)=0$.  

  If by such an enumeration, we cause an element of
  $(W_j-(Y^j\cup A))[s]$ to enter $A$, then we say L$_j$ does not {\em
    allow} P$_h^\alpha$ to add witnesses for any $\alpha$ such that
  $\alpha(j)=0$.  In addition, if the size of $(W_j-(Y^j\cup A))$ has
  become $N$, we say that L$_j$ {\em disallows for
    $N$}.


%
%

  {\bf Step 6} (outer splitting property): Let $X_i=\{x\ | \ x$ is
  ever a witness for any P$_{h,e}^\alpha$ with
  $\langle h,e\rangle<i\}$.  For each $x\in W_{i,s+1}$ such that $x$
  is not yet in $W_{f(i)}$ or $W_{g(i)}$, do the first of the
  following that applies:

  \begin{itemize}

  \item[(a)] If $x\in W_i-(X_i\cup A)$ and
    $|W_{f(i)}-(X_i\cup A)|<i+1$, put $x$ in $W_{f(i)}$.  Reset all
    P$_{h,e}^\alpha$ for $\langle h,e\>\geq i$.
  \item[(b)] If there is an $\langle h, e\rangle<i$ and any $\alpha$
    such that $x$ is an inactive witness for P$_{h,e}^\alpha$, and
    $W_{f(i)}$ does not already contain an inactive witness, then put
    $x$ into $W_{f(i)}$.
  \item [(c)] If $\Wfi$ contains no inactive witness and there is an
    $\langle h, e\rangle<i$ such that $x\in W_i\cap S_{h,e}^\alpha$
    and $W_{f(i)}\cap S_{h,e}^\alpha$ is empty, then put $x$ into
    $W_{f(i)}$.
  \item [(d)] If none of the above applies, put $x$ into $W_{g(i)}$.
  \end{itemize}

  {\bf Step 7} (coding D): If $k\in D_{s+1}-D_s$, enumerate one
  element from $F_k$ into $A_{s+1}$ such that it obeys the following:
  for each $M>0$ and $j+M=k$, the element is not one of the
  least 
  $M$ elements in $(W_j-(Y^j\cup A))[s]$, and for each $i\leq k$, the
  element is not one of the least $i+1$ elements in
  $W_{f(i)}-(X_i\cup A)$.

  \subsection{Verification}

  \begin{lemma}\label{semilowtwoprep} For each $j\in\omega$,
    $W_j\cap \ov A$ is infinite if and only if either
    $W_j\cap \ov A\cap Y^j$ is infinite or $W_j-(Y^j\cup A)$ has
    infinitely many expansionary stages.
  \end{lemma}

\begin{proof}

Note that the proof of this lemma is a very slight variation on the proof of Lemma \ref{semilowonefive}.

  If $W_j\cap \ov A$ is infinite then either $W_j\cap \ov A\cap Y^j$
  is infinite or $W_j-(Y^j\cup A)$ is infinite, in which case it has
  infinitely many expansionary stages.

  If $W_j\cap \ov A\cap Y^j$ is infinite, then $W_j\cap \ov A$ is
  infinite.  Suppose $W_j-(Y^j\cup A)$ has infinitely many
  expansionary stages.  We will show that for each $M$,
  $W_j-(Y^j\cup A)$ has at least $M$ elements.

  Induct on $M$.  Assume true for
  $M-1$.  

  No element in $(W_j-(Y^j\cup A))[s]$ ever enters $Y^j$ since new
  witnesses are chosen to be larger than $s$, and elements of
  $(W_j-(Y^j\cup A))[s]$ are smaller than $s$.%

  For $k\leq
    j$, only finitely many elements are ever put into $A$ by
    C$_k$.
    For $k>j$,
    C$_k$
    can only bring the size of $W_j-(Y^j\cup
    A)$ below $M$ if $M+j>k$, which happens finitely often.  Let
    $s_0$
    be a stage by which the least $M-1$
    elements in $(W_j-(Y^j\cup
    A))[s_0]$ never enter $A$ and after which no
    C$_k$
    enumerates any of the least $M$
    elements of $(W_j-(Y^j\cup
    A))[s]$ into $A$.  Let
    $s_1>s_0$
    be a stage after which Step 2 never acts for ($j,$
    $M-1$,
    $\beta$)
    for any $\beta$
    of length $j$.
    We may assume that at every stage $s>s_1$,
    the $M$th
    least element of $(W_j-(Y^j\cup
    A))[s]$ is a witness for some
    P$_{h,e}^\alpha$,
    else it would never enter $A$
    and we would be guaranteed to have at least $M$
    elements in $W_j-(Y^j\cup
    A)$, as desired.  Note that the length of
    $\al$ is greater than $j$ since the witnesses are not in $Y^j$.

  Let $s_2>s_1$
  be a stage such the $M$th
  least element in $(W_j-(Y^j\cup
  A))[s_2]$, called
  $x_M$,
  has the least use of any that will ever be in the $M$th
  position of any $(W_j-(Y^j\cup
  A))[s]$ for $s\geq s_2$.  Now, if $x_M$ never enters
  $A$
  by Step 5, we are done, so assume $x_M$
  enters $A$.
  When $x_M$
  enters $A$,
  all witnesses with equal or larger uses also enter $A$,
  and since the next element in the $M$th
  position is a witness, and it cannot have smaller use than $x_M$
  due to minimality, then the next element that enters the
  $M$th
  position has yet to become a witness.  When $x_M$
  enters $A$,
  L$_j$
  does not allow any P$_h^\alpha$
  to add witnesses for any $\alpha$
  such that $\alpha(j)=0$.
  Step 2 has finished acting for $j$
  and $M-1$,
  so it will not cause L$_j$
  to later allow.  The next time L$_j$
  allows will be by Step 1, which means $(W_j-(Y^j\cup
  A))$ has already reached a new expansionary stage.  Thus, at the
  first expansionary stage after $x_M$
  enters $A$,
  the only witness that could be in the $M$th
  position must be a witness for an $\alpha$
  with $\alpha(j)=1$.
  At that new expansionary stage, Step 1 inactivates the witnesses for
  $\alpha$.
  Thus the element in the $M$th
  position can never enter $A$,
  so $(W_j-(Y^j\cup A))$ has at least $M$
  elements. 

\end{proof}

We say that $\alpha$ is on the {\em true path} if $\alpha(j)=0$ if and
only if $W_j-(Y^j\cap \ov A)$ has infinitely many expansionary stages
during the construction.


\begin{lemma}\label{notsemilowonefive} Let $\alpha$ be on the true
  path.  Let $e$ be the least such that 
  either $\displaystyle \lim_s \ell_\alpha(e,s)$ does not exist or it exists and
  $\displaystyle \lim_s \ell_\alpha(e,s)\neq K^D(e)$.  Then
  P$_{h,e}^\alpha$ is reset finitely often ($n$ times) and
  $W_{h(r(\alpha,e,n))}$ is infinite if and only if
  $W_{r(\alpha,e,n)}\cap \ov A$ is finite.  Thus $\ov A$ is not
  semilow$_{1.5}$.
\end{lemma}

\begin{proof} Since $\alpha$ is on the true path, P$_{h,e}^\alpha$ is
  reset only finitely often by L$_j$, $j<h$.  P$_{h,e}^\alpha$ can
  only be reset by O$_i$ in Step 6 when $\langle h,e\rangle \geq i$, and each of
  these O$_i$ can only reset it finitely often, so it only gets reset
  finitely often by O$_i$.

  P$_{h,e}^\alpha$ can only be reset by P$_{h,e'}^\alpha$ for $e'<e$
  when $W_{h(r(\alpha,e',n'))}$ grows, where $n'$ is the number of
  times P$_{h,e'}^\alpha$ has been reset.  When it grows, we set
  $\ell_\alpha(e',s+1)=1$ from $\ell_\alpha(e',s)=0$, so if
  P$_{h,e'}^\alpha$ resets P$_{h,e}^\alpha$ infinitely often, then
  $\lim_s \ell_\alpha(e',s)$ does not exist, contradicting the
  assumption that it exists and equals $K^D(e')$.  Thus
  P$_{h,e}^\alpha$ is reset finitely often.

  {\em Case 1}: $W_{h(r(\alpha,e,n))}$ is infinite.

  {\em Case 1a}: $\Phi_e^D(e)$ diverges.  Thus,
  $\lim_s \ell_\alpha(e,s)\neq 0$.  If $S_{h,e}^\alpha$ is finite,
  we're done, since $W_{r(\alpha,e,n)}=S_{h,e}^\alpha$.  If
  $S_{h,e}^\alpha$ is infinite, then $\Phi_e^D(e)[s]$ converges
  infinitely often.  Any elements that we add when $\Phi_e^D(e)[s]$
  converges will be enumerated into $A$ when $\Phi_e^D(e)[t]$ diverges
  for $t>s$ because the use of each witness equals the use of the
  computation $\Phi_e^D(e)[s]$.  Thus
  $W_{r(\alpha,e,n)}\cap \ov A=S_{h,e}^\alpha\cap \ov A$ is finite.

  {\em Case 1b}: $\Phi_e^D(e)$ converges.  Thus,
  $\lim_s \ell_\alpha(e,s)\neq 1$.  If $W_{r(\alpha,e,n)}\cap \ov A$
  is infinite, then $\lim_s \ell_\alpha(e,s)$ cannot equal 0 because
  there will be infinitely many stages where Step 4 applies, setting
  $\ell_\alpha(e,s+1)$ to 1.  Thus, if $\Phi_e^D(e)\converges$, then
  $\lim_s \ell_\alpha(e,s)$ does not exist, but this is also
  impossible since the only time that $\ell_\alpha(e,s)$ changes from
  $1$ to $0$ is when $\Phi_e^D(e)[s]$ diverges or changes its
  computation, meaning that $\Phi_e^D(e)$ must diverge.

Thus, if $W_{h(r(\al,e,n))}$ is infinite, then $W_{r(\al,e,n)}\cap \ov A$ is finite.

  {\em Case 2}: $W_{h(r(\alpha,e,n))}$ is finite.  Note that
  $\lim_s \ell_\alpha(e,s)$ must exist because it can change to $1$
  only finitely often since $W_{h(r(\alpha,e,n))}$ is finite.

  {\em Case 2a}: $\Phi_e^D(e)$ diverges.  Thus,
  $\lim_s \ell_\alpha(e,s)=1$.  Note that at the final stage $s$ when
  $\ell_\alpha(e,s)$ changes from $0$ to $1$, $S_{h,e}^\alpha$
  contains an active witness not in $A$.  This means that
  $\Phi_e^D(e)[s]\converges$, else the active witness would have
  entered $A$.  Since $\Phi_e^D(e)\diverges$, then the computation
  $\Phi_e^D(e)[s]$ must eventually be injured, at which point
  $\ell_\alpha(e,t)$ will become $0$, contradicting that the limit is
  $1$.

  {\em Case 2b}: $\Phi_e^D(e)$ converges.  Thus,
  $\lim_s \ell_\alpha(e,s)=0$.  Because $\Phi_e^D(e)$ converges, if
  infinitely many elements enter $S_{h,e}^\alpha$, then infinitely
  many elements stay in $S_{h,e}^\alpha\cap \ov A$, since they only
  enter if $\Phi_e^D(e)[s]$ has its computation injured.  Thus it
  suffices to show that infinitely elements enter $S_{h,e}^\alpha$.
  To do so, we will show that at infinitely many stages, Step 3 acts.
  We already know that for almost all $s$, $\ell_\alpha(e,s)=0$ and
  $\Phi_e^D(e)[s]\converges$, so the first two bullet points of Step 3 are met
  for almost all $s$.

  To check that the third bullet point is met, note that for each
  $e'<e$, either $K^D(e')=1=\lim_s \ell_\alpha(e',s)$, in which case
  for almost all $s$, $e'$ does not prevent $e$ from acting, or
  $K^D(e')=0=\lim_s \ell_\alpha(e',s).$ In the latter case, there are
  two possibilities.  If for almost all $s$,
  $\Phi_{e'}^D(e')[s]\diverges$, then for almost all $s$, $e'$ does not
  prevent $e$ from meeting the third bullet point.  Otherwise, there
  are infinitely many stages such that $\Phi_{e'}^D(e')[s]\converges$,
  but since $\Phi_{e'}^D(e')\diverges$, the computation changes
  infinitely often, so eventually the computation will have changed at
  least $e$ times.  Thus, $e'$ only prevents $e$ from adding a new
  witness for finitely many stages.

  Now that we have seen that for almost all $s$, the three bullet
  points are met, we must still show that $\alpha$ is {\em allowed} to
  add a witness at a stage where it is able to act.

  {\em Claim}: For every $\alpha$ on the true path, $\alpha$ is
  allowed to add witnesses infinitely often.

For the proof of the claim, see Lemma \ref{notsemi}, replacing the reference to Lemma \ref{semilowonefive} with Lemma \ref{semilowtwoprep}.

  Now wait until a stage $s_0$ such that the approximation to the true
  path never goes to the left of $\alpha$ and the bullet points do not
  prevent P$_{h,e}^\alpha$ from acting after stage $s_0$.  Let
  $s_1>s_0$ be any $\alpha$-stage, and $s_2\geq s_1$ the next stage
  that $\alpha$ is allowed to add witnesses.  Then $s_2$ is either an
  $\alpha$-stage or $\delta_{s_2}$ is to the right of $\alpha$, so we
  may perform Step 3 for P$_{h,e}^\alpha$.  Thus, there are infinitely
  many stages at which we will add witnesses to $S_{h,e}^\alpha$, so
  it will be infinite and have infinite intersection with $\ov A$.

\end{proof}

\begin{lemma}\label{zerodouble} $\bf 0''$ can compute
  $\{ j\ | \ W_j\cap \ov A$ is infinite$\}$.
\end{lemma}

\begin{proof}
  By Lemma \ref{semilowtwoprep}, it suffices to show that $\bf 0''$
  can compute $\{j \mid W_j\cap \ov A\cap Y^j$ is infinite$\}$, since
  $\bf 0''$ can determine if the set of expansionary stages of
  $W_j-(Y^j\cup A)$ is infinite since the set of expansionary stages
  is a \ce set.

  We will show that $\ov A\cap Y^j$ is a \ce set and that there is a
  $\bf 0''$-computable function $m(j)$ such that
  $W_{m(j)}=^*\ov A\cap Y^j$, so $W_j\cap\ov A\cap Y^j$ is infinite if
  and only if $W_j\cap W_{m(j)}$ is infinite, which is
  $\0''$-computable.

  Note that $\bf 0''$ can determine the true path on the tree of
  strategies.  Consider each $Y_{h,e}^\al$ that makes up $Y^j$.  If
  $\alpha$ is to the right of the true path, then P$_{h,e}^\alpha$
  gets reset infinitely often, so the set of elements in
  Y$_{h,e}^\alpha\cap \ov A$ is the \ce set given by the set of all
  $x$ in Y$_{h,e}^\alpha\cap \ov A$ at a stage when P$_{h,e}^\alpha$
  is reset.  Put these into $W_{m(j)}$.

  If $\alpha$ is to the left of the true path, there are only finitely
  many $\alpha$-stages.  Step 3 is either allowed to add witnesses at
  the final $\alpha$-stage or at one later stage.  Thus
  S$_{h,e}^\alpha$ only gets finitely many elements, so it may be
  ignored.

  For each $\alpha$ on the true path of length $i\leq j$, use
  $\bf 0''$ to find the least $e$ such that
  $\lim_s \ell_\alpha(e,s)\neq K^D(e)$.  For $e'<e$, note that if
  there are infinitely many witnesses for P$_{h,e'}^\alpha$, then
  $\ell_\alpha(e',s)$ must be 0 infinitely often, so its limit must be
  0.  We also have that $\Phi^D_{e'}(e')[s]$ must converge infinitely
  often for us to add infinitely many witnesses, so it must also
  diverge infinitely often so that
  $K^D(e')=\lim_s \ell_\alpha(e',s)=0$.  Thus, once P$_{h,e'}^\alpha$
  stops being reset, which we know must happen by the proof of Lemma
  \ref{notsemilowonefive}, all new witnesses eventually enter $A$.
  Thus, P$_{h,e'}^\alpha$ only contributes finitely much to
  $Y^j\cap \ov A$ and can be ignored.

  Now consider $e'>e$.  If P$_{h,e}^\alpha$ resets P$_{h,e'}^\alpha$
  infinitely often, then the contribution of P$_{h,e'}^\alpha$ to
  $Y^j\cap \ov A$ is given by the set of elements that are witnesses
  at stages when it gets reset, and these can be added to $W_{m(j)}$.
  Suppose P$_{h,e}^\alpha$ resets P$_{h,e'}^\alpha$ only finitely
  often.  First, consider the case where $\lim_s\ell_\alpha(e,s)=1$. Then
  $\Phi_e^D(e)\diverges$ by the choice of $e$.  However, if $\lim_s\ell_\alpha(e,s)=1$, then when $\ell_\alpha(e,s)$ gets defined as $1$, it contains an active witness, and the witness cannot enter $A$ else it would reset $e'$, so $\Phi_e^D(e)\converges$.  This is a contradiction, so $\lim_s\ell_\alpha(e,s)\neq 1$.  Since P$_{h,e}^\alpha$ resets P$_{h,e'}^\alpha$ only
  finitely often, we must have that $\lim_s \ell_\alpha(e,s)=0$.
  Thus, $\Phi_e^D(e)\converges$, so $\Phi_e^D(e)[s]$ changes
  computations some finite number of times. By the final bullet of
  Step 3, only finitely many $e'>e$ will ever be able to add
  witnesses, and will only be able to do so finitely often.  Thus
  after a finite stage, no P$_{h,e'}^\alpha$ can add any witnesses,
  for $e'>e$, so these collectively contribute only finitely many
  elements to $Y^j\cap \ov A$ and can be ignored.

  For $e$ itself, $S_{h,e}^\alpha\cap \ov A$ is infinite if and only
  if $W_{h(r(\alpha,e,n))}$ is finite, for $n$ the number of times
  P$_{h,e}^\alpha$ is reset, by Lemma \ref{notsemilowonefive}.  We can
  ask $\0''$ to determine this.  If we know that
  $S_{h,e}^\alpha\cap \ov A$ is finite, we can ignore it.  If we know
  it is infinite, then $W_{h(r(\alpha,e,n))}$ is finite, so $\lim_s \ell_\alpha(e,s)=0$, and thus
  $\Phi^D_e(e)\converges$ and so almost all elements that we put into
  $S_{h,e}^\alpha$ remain in $\ov A$, so we can put them all into
  $W_{m(j)}$.
\end{proof}

%
%
%
%
%

\begin{lemma} The following hold for all $i\in\omega$:\begin{itemize}
  \item[(a)] $W_i=W_{f(i)}\sqcup W_{g(i)}$,
  \item[(b)] $\Wfi\cap \ov A=^* \es$, and
  \item[(c)] $W_i\cap \ov A$ infinite $\implies$ $\Wfi\cap \ov A$ is
    nonempty.
  \end{itemize}
\end{lemma}

\begin{proof}
  By construction, we put every element of $W_i$ into either $\Wfi$ or
  $\Wgi$, but not both, so part (a) holds.

  {\em Proof of (b)}: We may add elements to $\Wfi$ by Steps 6a, 6b,
  and 6c.  The set $(\Wfi\cap\ov A)-X_i=\Wfi-(X_i\cup A)$ must be
  finite because only Step 6a can add elements to it and only when its
  size is less than $i+1$, so at most $i+1$ elements will stay in
  $W_{f(i)}-(X_i\cup A)$
  forever.  

  To see that $\Wfi\cap\ov A\cap X_i$ is finite, we look at Steps 6b
  and 6c.  
  If Step 6b adds an inactive witness to $W_{f(i)}$, it will never be
  able to do so again, as inactive witnesses remain inactive witnesses
  forever for the remainder of the construction.  Step 6c allows us to
  take an element from $W_i\cap S_{h,e}^\alpha\cap\ov A$ only when
  there is not already an element in
  $\Wfi\cap S_{h,e}^\alpha\cap\ov A$, for $\langle h,e\rangle<i$.
  Thus, even if we put infinitely many elements in by Step 6c, only
  finitely many remain forever.

  {\em Proof of (c)}: Suppose $W_i\cap \ov A$ is infinite.  Suppose
  there are eventually $i+1$ elements that enter $\Wfi$ by Step 6a.
  When we add elements by Step 6a, they are taken from
  $W_i-(X_i\cup A)$, which means that they cannot be put into $A$ by
  action of P$_{h,e}^\alpha$ for any $\langle h,e\rangle<i$.  When we
  put the elements into $\Wfi$, we also reset P$_{h,e}^\alpha$ for
  each $\langle h,e\rangle\geq i$, so they cannot enter $A$ from
  action of these requirements, either.  They can only be enumerated
  into $A$ by C$_k$, for $k<i$, which each only act once.  Thus, at
  least one element will remain in $\Wfi\cap \ov A$.  If there are not
  eventually $i+1$ elements that enter $\Wfi$ by Step 6a, then
  $W_i-(X_i\cup A)$ must be finite, so $W_i\cap \ov A\subseteq^* X_i$,
  and so $W_i\cap\ov A\cap X_i$ is infinite.  If Step 6b ever acts by enumerating into $W_{f(i)}$,
  then $\Wfi$ contains an inactive witness, which will never enter
  $A$.  Otherwise, we know there is some $\langle h,e\rangle<i$ such
  that there are infinitely many witnesses for P$_{h,e}^\alpha$
  contained in $W_i\cap \ov A$, by the pigeonhole principle.  Since
  Step 6b never acts, all of the witnesses are active when they enter
  $W_i$.  Suppose $\Wfi\cap \ov A$ is empty.  Then every time Step 6c
  acts by adding an active witness to $\Wfi$, that element later
  enters $A$ when the current computation $\Phi_e^D(e)[s]$ changes
  because $D$ changes below the use.  For each $x$ that is an active
  witness for P$_{h,e}^\alpha$ when it enters $W_i$, it either enters
  $W_{f(i)}$, in which case it later enters $A$, or there is already
  an active witness $y$ for P$_{h,e}^\alpha$ in $\Wfi\cap \ov A$ when
  it enters $W_i$.  In that case, we know that $y$ goes into $A$, at
  which stage $x$ must go into $A$ as well because they have the same
  use.  
  Thus, if $\Wfi\cap \ov A$ is empty, then $W_i\cap \ov A$ is finite,
  proving part
  (c). 
\end{proof}

\begin{lemma} $D\leT A$.
\end{lemma}

\begin{proof}
  To determine if $k\in D$, ask if $A$ contains any elements of $F_k$.
  $F_k\cap A$ is nonempty if and only if $k\in D$.  This is because
  only C$_k$ can enumerate elements of $F_k$ into $A$, and C$_k$ is
  always able to act because $F_k$ has more elements in it than are
  prohibited.  To see that $F_k$ has more elements than are
  prohibited, recall that for each $M>0$ and $j+M=k$, we prohibit the
  least $M$ elements in $(W_j-(Y^j\cup A))[s]$, and for each
  $i\leq k$, we prohibit the least $i+1$ elements in
  $W_{f(i)}-(X_i\cup A)$.  For $k=0$, one element is prohibited and
  $F_0$ has two elements.  For $k>0$, the first clause prohibits less
  than $2k^2$ elements because
  $\displaystyle\sum_{M=1}^k M=\frac{(k+1)k}2=\frac{k^2}2+\frac
  k2<2k^2$.
  The second clause prohibits less than $4k^2$ elements because
  $\displaystyle\sum_{i=0}^k
  i+1=\frac{(k+1)(k+2)}2=\frac{k^2}2+\frac{3k}2+1<4k^2$.
  Thus fewer than $6k^2$ elements are prohibited and $F_k$ contains
  $6k^2$ elements.

\end{proof}

\begin{lemma} $A\leT D$.
\end{lemma}

\begin{proof}

  $\omega^{[0]}\cap A\leT D$ since no element from $\omega^{[0]}$
  enters $A$ unless it is in $F_k$ and $k$ enters $D$.

  For $\omega-\omega^{[0]}$, $x\in A$ if and only if $x$ becomes a
  witness and then enters $A$.  The set of witnesses is computable
  because if $x$ is not a witness by stage $x$, it will never become a
  witness. If $x$ is a witness, we ask what its use $u$ is as a
  witness.  Let $s_x$ be such that $D\res u=D_{s_x}\res u$.  Then
  $x\in A$ if and only if $x\in A_{s_x}$.
\end{proof}

This concludes the proof of the theorem.
\end{proof}

\section{Non-Low$_2$ degrees and sets whose complement is not
  semilow$_2$}

In this section we will provide an index set argument that every
nonlow$_2$ c.e\ degree contains a c.e.\ set whose complement is
not semilow$_2$.  Soare \cite[IV 4.11]{Soare:87} shows via an index
set argument that every non-low c.e\ degree contains a c.e.\ set
whose complement is not semilow.

Recall that $Inf$ is the set of indices for infinite c.e.\ sets.
$Inf^B$ is $\Pi^0_2$ complete. $Inf^B = \{e \mid W^B_e$ is infinite$\}$
and $\overline{B''} \leq_1 Inf^B$.  Let $Inf(X) = \{e \mid W_e \cap X $
is infinite$\}$. Recall $\overline{A}$ is semilow$_{1.5}$ if
$Inf(\overline{A}) \leq_1 Inf$ and $\overline{A}$ is semilow$_2$ if
$Inf(\overline{A}) \leq_T Inf$

\begin{theorem}\label{notlow22}
  For every \ce set $B$, there is a \ce $A\equivT B$ such that 
  $Inf^B\leq_1 Inf(\ov A)$.
\end{theorem}

\begin{Cor}
  If $B$ is nonlow$_2$, then there exists a \ce $A\equivT B$ such that
  $\ov A$ is not semilow$_2$ (and not semilow$_{1.5}$).
\end{Cor}

\begin{proof} Let $B$ be nonlow$_2$.  Then there exists a \ce
  $A\equivT B$ such that $Inf<_T Inf^B\leq_1 Inf (\ov A)$, so $\ov A$
  is not semilow$_2$ (and not semilow$_{1.5}$).
\end{proof}

\begin{proof}[Proof of Theorem~\ref{notlow22}]
  Let $\{\Phi_i\}_{i\in \omega}$ be a listing of all Turing
  functionals.  Let $\{B_s\}_{s\in \omega}$ be a computable
  enumeration of $B$.  We may assume that if
  $\Phi_i^B(x)[s]\converges\neq \Phi_i^B(x)[t]\converges$, then there
  is a stage between $s$ and $t$ such that the computation diverges.
  We may also assume that at each stage $s$, there is at most one pair
  $\<i,x\>$ such that $\Phi^B_i(x)[s]\converges$ and
  $\Phi^B_i(x)[s-1]\diverges$.  The use function of the computation
  $\Phi^B_i(x)[s]$ is denoted $\phi^B_i(x)[s]$ and is the maximal
  element of $B_{i,s}$ seen in the computation.

  {\bf Construction}:

  {\em Stage $s+1$}:

  \emph{Step 1}: If $b\in B_{s+1}-B_s$, enumerate all marked elements
  into $A_{s+1}$ whose markers have uses $u\geq b$.  Whenever an
  element is enumerated into $A$, its marker is removed.

  \emph{Step 2}: If $\Phi_i^B(x)[s+1]\converges$, and there is no
  current marker $M_{\<i,x\>}$, then choose the least element
  $y\in \ov A_s$, $y>s+1$, without a marker and place marker
  $M_{\<i,x\>}$ on it.  The use of this marker is $\phi_i^B(x)[s+1]$.

  \emph{Step 3:} Let $b\notin B_{s+1}$ be the least such that there is
  no current marker $\Gamma_b$.  Place $\Gamma_b$ on the least
  $y\in\ov A_s$, $y>s+1$, without any marker, and let $b$ be the use
  of this marker.
  
  End construction.\\

  \begin{lemma}
    $A\leT B$.
  \end{lemma}

  \begin{proof} To determine if $y\in A$, first run the construction
    until stage $y$ to see if $y$ ever has a marker.  If not, then
    $y\notin A$.  If so, then if the marker was added at stage $s$ and
    the use of the marker on $y$ is $u$, ask if $B$ ever changes at or
    below $u$ after stage $s$.  If so, then $y\in A$.  Otherwise
    $y\notin A$.
  \end{proof}

  \begin{lemma}
    $B\leT A$.
  \end{lemma}

\begin{proof}
  Recall that the use of any marker $\Gamma_b$ is $b$.  Thus, once
  $B_s$ has settled up through $b$, no current or future marker
  $\Gamma_b$ will ever enter $A$.  To determine if $b\in B$, run the
  construction until either $\Gamma_b$ is placed on an element in
  $\ov A$, in which case $b\notin B$, or $b$ enters $B_s$.  If
  $b\notin B$, then eventually a marker $\Gamma_b$ will appear in Step
  3 on an element in $\ov A$, as desired.
\end{proof}

\begin{lemma}
  Let $W_{f(i)}=\{y\mid (\exists s)(\exists x) [M_{\<i,x\>}$ is
on $y$ at stage $s]\}.$ $W_i^B$ is infinite if and only if
$W_{f(i)}\cap \ov A$ is infinite.
\end{lemma}

\begin{proof}
  If $W_i^B$ is infinite, then there exist infinitely many $x$ such
  that $\Phi_i^B(x)\converges$.  For such $x$, for almost all $s$,
  $\Phi_i^B(x)[s]\converges$ and $B_s$ never later changes at or below
  the use $\phi_i^B(x)[s]$.  For each such $x$, $M_{\<i,x\>}$ will be
  undefined before the final computation $\Phi_i^B(x)[s]$ converges,
  because any previous computation would have been injured below the
  use, causing the marker to be removed.  When the final computation
  appears, the marker will be placed by Step 2 on an element in
  $\ov A_s$, and this element will never enter $A$ since nothing
  enters $B$ at or below the use.  Thus for each of these infinitely
  many $x$ values, the final resting place of $M_{\<i,x\>}$ is in
  $\ov A$, and thus $W_{f(i)}\cap \ov A$ is infinite.

  If $W_i^B$ is finite, then for almost all $x$, $\Phi_i^B(x)$
  diverges, which means that for almost all $x$, every computation
  $\Phi_i^B(x)[s]$ that converges is injured below the use.  Thus, for
  almost all $x$, whenever $M_{\<i,x\>}$ is placed on an element, that
  element later enters $A$.  There are only finitely many $x$ values
  where $M_{\<i,x\>}$ is ever on an element that remains outside of
  $A$, and once a marker is placed on an element outside $A$, it
  remains there forever, and there is always only one current
  $M_{\<i,x\>}$.  Thus, only finitely many elements of $W_{f(i)}$ are
  in $\ov A$.
\end{proof}

\end{proof}

\bibliographystyle{plain}
\bibliography{paper}

\begin{thebibliography}{10}

\bibitem{mr95f:03064}
Peter Cholak.
\newblock Automorphisms of the lattice of recursively enumerable sets.
\newblock {\em Mem. Amer. Math. Soc.}, 113(541):viii+151, 1995.

\bibitem{mr2003h:03063}
Peter Cholak and Leo~A. Harrington.
\newblock On the definability of the double jump in the computably enumerable
  sets.
\newblock {\em J. Math. Log.}, 2(2):261--296, 2002.

\bibitem{MR3125901}
Rodney~G. Downey, Carl~G. Jockusch, Jr., and Paul~E. Schupp.
\newblock Asymptotic density and computably enumerable sets.
\newblock {\em J. Math. Log.}, 13(2):1350005, 43, 2013.

\bibitem{MR3003266}
Rachel Epstein.
\newblock The nonlow computably enumerable degrees are not invariant in
  {$\mathcal{E}$}.
\newblock {\em Trans. Amer. Math. Soc.}, 365(3):1305--1345, 2013.

\bibitem{MR1640265}
Leo Harrington and Robert~I. Soare.
\newblock Definable properties of the computably enumerable sets.
\newblock {\em Ann. Pure Appl. Logic}, 94(1-3):97--125, 1998.
\newblock Conference on Computability Theory (Oberwolfach, 1996).

\bibitem{Harrington.Soare:96}
Leo~A. Harrington and Robert~I. Soare.
\newblock The ${\Delta}\sp 0\sb 3$-automorphism method and noninvariant classes
  of degrees.
\newblock {\em J. Amer. Math. Soc.}, 9(3):617--666, 1996.

\bibitem{Lachlan:68}
A.~H. Lachlan.
\newblock Degrees of recursively enumerable sets which have no maximal
  supersets.
\newblock {\em J. Symbolic Logic}, 33:431--443, 1968.

\bibitem{Maass:83}
W.~Maass.
\newblock Characterization of recursively enumerable sets with supersets
  effectively isomorphic to all recursively enumerable sets.
\newblock {\em Trans. Amer. Math. Soc.}, 279:311--336, 1983.

\bibitem{MR1933396}
Russell Miller.
\newblock Orbits of computably enumerable sets: low sets can avoid an upper
  cone.
\newblock {\em Ann. Pure Appl. Logic}, 118(1-2):61--85, 2002.

\bibitem{Shoenfield:76}
Joseph~R. Shoenfield.
\newblock Degrees of classes of recursively enumerable sets.
\newblock {\em J. Symbolic Logic}, 41:695--696, 1976.

\bibitem{Soare:74}
Robert~I. Soare.
\newblock Automorphisms of the lattice of recursively enumerable sets {I}:
  maximal sets.
\newblock {\em Ann. of Math. (2)}, 100:80--120, 1974.

\bibitem{Soare:82}
Robert~I. Soare.
\newblock Automorphisms of the lattice of recursively enumerable sets {II}: low
  sets.
\newblock {\em Ann. Math. Logic}, 22:69--107, 1982.

\bibitem{Soare:87}
Robert~I. Soare.
\newblock {\em Recursively Enumerable Sets and Degrees}.
\newblock Perspectives in Mathematical Logic, Omega Series. Springer--Verlag,
  Heidelberg, 1987.

\end{thebibliography}

\end{document}